\documentclass[10pt]{article}
\usepackage[utf8]{inputenc}

\usepackage{amsfonts,amsmath, amssymb,amsthm,amscd}
\usepackage{tikz}
\usepackage{color}
\usepackage{mathtools}
\usepackage{dsfont}
\usepackage{comment}
\usepackage{geometry}

\newtheorem{theorem}{Theorem}
\newtheorem{proposition}{Proposition}
\newtheorem{corollary}{Corollary}

\newtheorem{lemma}{Lemma}
\newtheorem{definition}{Definition}

\newtheorem{remark}{Remark}
\newtheorem{example}{Example}

\newcommand{\E}{{\mathbb E }}
\newcommand{\C}{{\mathbb C }}
\newcommand{\R}{{\mathbb R }}

\renewcommand{\Im}{{\text{Im} }}
\renewcommand{\i}{{\text{i} }}
\renewcommand{\P}{{\mathbb P}}
\newcommand{\p}{{\mathbf p}}

\allowdisplaybreaks
\begin{document}
	\begin{minipage}{0.85\textwidth}
	\end{minipage}
	\begin{center}
		\large\bf
Central limit theorem for eigenvalue statistics of sample covariance matrix with random population 
	\end{center}
	
	\begin{center} 
		\begin{minipage}{0.3\textwidth}
			\begin{center}
				Ji Oon Lee\\
				\footnotesize 
				{Korea Advanced Institute of Science and Technology}\\
				{\it jioon.lee@kaist.edu}
			\end{center}
		\end{minipage} 
		\begin{minipage}{0.3\textwidth}
			\begin{center}
				Yiting Li  \\
				\footnotesize 
				{Korea Advanced Institute of Science and Technology}\\
				{\it yitingli@kaist.ac.kr}
			\end{center}
		\end{minipage} 
	\end{center}

	\begin{abstract}
 Consider the sample covariance matrix
 \begin{align} \label{eq87www}
 	\Sigma^{1/2}XX^T\Sigma^{1/2}
 \end{align}
 where $X$ is an $M\times N$ random matrix with independent entries and $\Sigma$ is an $M\times M$ diagonal matrix.   It is known that if $\Sigma$ is deterministic, then the fluctuation of
 $$\sum_if(\lambda_i)$$
 converges in distribution to a Gaussian distribution. Here $\{\lambda_i\}$ are eigenvalues of \eqref{eq87www} and $f$ is a good enough test function. In this paper we consider the case that $\Sigma$ is random and show that the fluctuation of 
 $$\frac{1}{\sqrt N}\sum_if(\lambda_i)$$
 converges in distribution to a Gaussian distribution. This phenomenon implies that the randomness of $\Sigma$ decreases the correlation among $\{\lambda_i\}$.
	\end{abstract}

\tableofcontents

\section{Background}

Consider the sample covariance matrix 
\begin{align}\label{eqn30}
\Sigma^{1/2}XX^T\Sigma^{1/2}
\end{align} 
where $X$ is an $M\times N$ random matrix with iid centered entries and the population $\Sigma $ is an $M\times M$ diagonal matrix with nonnegative entries. Assume that the empirical distribution of $\Sigma$ converges to a deterministic probability measure $\sigma$ and also assume that
$$\frac{M}{N}\to\gamma_0\in(0,\infty).$$ 

For the case $\Sigma=I$,    Marchenko and Pastur   \cite{Marchenko+Pastur} proved that the empirical distribution of \eqref{eqn30} converges to a deterministic probability measure with parameter $\gamma_0$. This measure is called the Marchenko-Pastur distribution and denoted by $\mu_{MP,\gamma_0}$. For the case $\Sigma\ne I$, it is known that empirical distribution of \eqref{eqn30} converges to the so called multiplicative free convolution of $\sigma$ and $\mu_{MP,\gamma_0}$, denoted by $\sigma\boxtimes\mu_{MP,\gamma_0}$. See \cite{Bai+Silverstein_spectral_analysis,Voiculescu,Voiculescu+Dykema+Nica}.

To know how fast the empirical distribution of \eqref{eqn30} converges to $\sigma\boxtimes\mu_{MP,\gamma_0}$, people study the asymptotic behavior of the fluctuation of the linear statistics:
\begin{align}\label{eqn31} 
\sum_i f(\lambda_i)-M\int f(t)d(\sigma\boxtimes\mu_{MP,\gamma_0})(t)
\end{align}
where $\{\lambda_i\}$ are eigenvalues of \eqref{eqn30} and $f$ is a test function.  The central limit theorem (CLT) of \eqref{eqn31} was first proved by Johansson \cite{Johansson} for Wishart matrices where $\Sigma=I$ and $X$ have Gaussian entries. He proved that \eqref{eqn31} converges to a Gaussian distribution. Note that this is different from the classic CLT for random variables which has the coefficient $\frac{1}{\sqrt N}$. This phenomenon  showed that the eigenvalues $\{\lambda_i\}$ have very strong correlation. Bai and Silverstein \cite{Bai+Silverstein} proved the CLT of \eqref{eqn31} for general $X$ and $\Sigma$  with a condition on the fourth moment of $X_{ij}$. This condition was later removed by the work of Pan and Zhou \cite{Pan+Zhou}.
In a series of work \cite{Bai+Wang+Zhou,lytova+pastur,Najim+Yao,M.Shcherbina}  the regularity condition of the test function was weakened. In \cite{Liu+Hu+Bai+Song} the CLT was proved in the case when the entries of $\Sigma$ do not have a uniform bound. In \cite{Li+Schnelli+Xu} the mesoscopic CLT was obtained.

All the work cited above assume that $\Sigma$ is a deterministic matrix. Kwak, the first author and Park \cite{Kwak+Lee+Park} considered the largest eigenvalue of \eqref{eqn30} for both deterministic and random $\Sigma$ and proved that the fluctuation of the largest eigenvalue converges to a deterministic probability measure.

In this paper, we consider the sample covariance matrix with random population and prove the CLT of \eqref{eqn31}. In fact, for a good enough test function $f$ we prove that the rescaled fluctuation
\begin{align}\label{eqn32}
	\frac{1}{\sqrt N}\Big(\sum_if(\lambda_i)-M\int f(t)d(\sigma\boxtimes\mu_{MP,\gamma_0})(t)\Big)
\end{align}
converges in distribution to a centered Gaussian distribution. Different from the case of deterministic $\Sigma$, we have the coefficient $\frac{1}{\sqrt N}$ in \eqref{eqn32}. This implies that the randomness of $\Sigma$ decreases the correlation among the eigenvalues. 

The similar phenomenon has been observed for the deformed Wigner matrix $W+V$ where $W$ is a Wigner matrix and $V$ is a  diagonal matrix. Suppose $\{\mu_i\}$ are eigenvalues of $W+V$. Ji and the first author \cite{Ji+Lee} proved that if $V$ is deterministic, then the fluctuation of
\begin{align}\label{eqn33}
\sum_i f(\mu_i)
\end{align}
converges to a Gaussian distribution and on the other hand, if $V$ is random, then the fluctuation of \eqref{eqn33} multiplied by $\frac{1}{\sqrt N}$ converges to a Gaussian distribution.

\section{Model and main results}

Define
$$\C_+=\{x+\i y|x\in\R,y>0\}.$$
For a probability measure $\pi$, define its Stieltjes transform by
$$m_\pi(z)=\int \frac{d\pi(t)}{t-z},\quad\forall z\not\in\text{supp}(\pi).$$
If $\pi$ has density $\rho_\pi(t)$ on an open interval $I$ and $\rho_\pi(t)$ is continuous on $I$, then by the Sokhotski–Plemelj theorem, 
	$$\rho_\pi(x)=\frac{1}{\pi}\lim\limits_{y\to0+} \Im m_\pi(x+\i y)\quad\forall x\in I.$$ 

\begin{lemma}\label{lemma:free_convolution}
	Let $\pi$ be a compactly supported probability measure on $\R$, and let $r>0$.
\begin{itemize}
	\item 
	For each $z\in\C_+$ there is a unique $m=m(z)\in\C_+$ satisfying
\begin{align}\label{self_consistent_eq}
\frac{1}{m}=-z+r\int \frac{t}{1+m t}d\pi(t).
\end{align}
	Moreover, $m(z)$ is the Stieltjes transform of a probability measure with compact
	support  in $[0,\infty)$. This probability measure is the multiplicative free convolution of
	$\pi$ and the Marchenko–Pastur law with dimensional ratio $r$:
	$$\pi \boxtimes\mu_{MP,r}$$
	where $\mu_{MP,r}$ denotes the Marchenko–Pastur law with dimensional ratio $r$.
	\item
	There is a continuous nonnegative function $w(x)$ defined on $(0,\infty)$ such that
\begin{align}\label{eqn1}
\pi\boxtimes\mu_{MP,r}=(1-r)^+\delta_0+w(x)dx.
\end{align}
	Moreover, $w(x)$ is analytic wherever it is positive.
\end{itemize}	
\end{lemma}
\begin{proof}
The first result is Lemma 2.2 in \cite{Knowles+Yin}.	The second result can be found on Page 2271 of \cite{Hachem+Hardy+Najim}.
\end{proof}

\begin{definition}\label{def:model_of_X}
Suppose	$X=(X_{ij})$ is an $M\times N$ matrix where $M=M(N)$ and
	\begin{itemize}
		\item $\{X_{ij}\}$ are iid real-valued random variables such that $\E[X_{ij}]=0$, $\E[X_{ij}^2]=\frac{1}{N}$
		
		\item for any $p\in\{1,2,\ldots\}$ there exists $C_p>0$ such that 
	\begin{align}\label{eq8}
	\E[|\sqrt NX_{ij}|^p]\le C_p,\quad\forall i,j,N
	\end{align}
\item there exist constants $\gamma_0\in(0,\infty)\backslash\{1\}$ and $c_0>0$ such that
\begin{align}\label{eqn18}
	|\frac{M}{N}-\gamma_0|\le N^{-\frac{1}{2}-c_0}
\end{align}
 Without loss of generality we assume $c_0<0.01$.
	\end{itemize}
\end{definition}

It is easy to see that if $\gamma_0\in(0,1)$, then $M<N$ when $N$ is large enough; if $\gamma_0>1$, then $M>N$ when $N$ is large enough. By \eqref{eq8}, for any (small) $\epsilon>0$ and (large) $D>0$, if $N>N_0=N_0(\epsilon,D)$ then
\begin{align}\label{eq44}
\P(|X_{ij}|\le N^{\epsilon-\frac{1}{2}},\quad\forall i,j)>1-N^{-D}.
\end{align}

\begin{definition}	Suppose $l\in(0,1)$ is a constant and $\nu$ is a probability measure with $\text{supp}(\nu)=[l,1]$. Let
	$\Sigma=\text{diag}(\sigma_1,\ldots,\sigma_M)$ be an $M\times M$   diagonal matrix  such that the entries $\sigma_1$,\ldots,$\sigma_M$ are iid with distribution $\nu$. Moreover we assume that $\Sigma$ is independent of $X$.
\end{definition}

By linear algebra, the set of eigenvalues of $X^T\Sigma X$ differs from the set of eigenvalues of $\Sigma^{1/2}XX^T\Sigma^{1/2}$ by $|M-N|$ zeros.

Suppose $\mu_\Sigma$ is the empirical measure of $\Sigma$:
$$\mu_\Sigma=\frac{1}{M}\sum_{i=1}^M\delta_{\sigma_i}.$$
\begin{lemma}\label{lemma:convergence}
	\begin{itemize}
		\item Almost surely, $\mu_\Sigma$ converges weakly to $\nu$. In other words, if $f(t)$ is  bounded and continuous, then almost surely we have 
\begin{align}\label{eq2}
\frac{1}{M}\sum_{i=1}^Mf(\sigma_i)\to \int f(t)d\nu(t).
\end{align}
		\item Suppose $\lambda_1\ge\cdots\ge\lambda_N\ge0$ are the   eigenvalues of  $X^T\Sigma X$. Let $\lambda_i=0$ for $M\wedge N<i\le M\vee N$. Almost surely, the empirical measure 
		$$\frac{1}{N}\sum_{i=1}^N\delta_{\lambda_i}$$
		 converge weakly to $$\mu_{fc}:=\nu\boxtimes\mu_{MP,\gamma_0}.$$
	\end{itemize}
\end{lemma}
\begin{proof}
	The first conclusion is from the Glivenko–Cantelli Theorem. See Theorem 2.4.7 of \cite{Durrett}.
	
	  	The second conclusion can be induced from the first paragraph and the Equation (1.3) of \cite{Silverstein+Choi}.

\end{proof}

\begin{definition}
	\begin{itemize}
		\item Let $\rho_{fc}(x)$ be the density function of the absolute continuous part of $\mu_{fc}$. According to Lemma \ref{lemma:free_convolution}, $\rho_{fc}(x)$ is continuous and
		$$\mu_{fc}=(1-\gamma_0)^+\delta_0+\rho_{fc}(x)dx$$
\item Similarly, denote $$\hat\mu_{fc}:=\mu_\Sigma\boxtimes \mu_{MP,M/N}$$ and let $\hat\rho_{fc(t)}$ be the continuous function defined on $(0,\infty)$ such that
$$\hat\mu_{fc}=(1-\frac{M}{N})^+\delta_0+\hat\rho_{fc}(x)dx.$$
		\item  Let $[L_-,L_+]$ be the smallest interval containing the support of $\rho_{fc}$. Let  $[\hat L_-,\hat L_+]$ be the smallest interval containing the support of $\hat\rho_{fc}(x)$.
	\end{itemize}
\end{definition}
By Proposition 2.4 of \cite{Hachem+Hardy+Najim} and the assumption $\gamma_0\ne1$, we have   $L_->0$ and $\hat L_->0$ when $N$ is large enough.

\begin{theorem}\label{thm:main_thm}
Suppose 
\begin{align}\label{basic_assumption}
\lim\limits_{N\to\infty}\P\Big(|\hat L_--L_-|<\epsilon\quad\text{and}\quad |\hat L_+-L_+|<\epsilon\Big)\to1\quad\quad\forall \epsilon>0.
\end{align}	
	
	If $f(x)$ is a function analytic in a neighborhood of $[L_-,L_+]$, then as $N\to\infty$, 
	$$\frac{1}{\sqrt N}\sum_{i=1}^Nf(\lambda_i)-\sqrt N\int f(t)d\mu_{fc}(t)$$
	converges in distribution to a centered Gaussian distribution whose variance is
\begin{multline*}
-\frac{1}{4\pi^2}\oint_\Gamma\oint_\Gamma\int_l^1\frac{t^2m_{fc}'(\xi_1)m_{fc}'(\xi_2)f(\xi_1)f(\xi_2)}{(1+tm_{fc}(\xi_1))(1+tm_{fc}(\xi_2))}d\nu(t)d\xi_1d\xi_2
+\frac{1}{4\pi^2\gamma_0^2}\Bigg(\oint_\Gamma f(\xi)m_{fc}'(\xi)(\xi+\frac{1}{m_{fc}(\xi)})d\xi\Bigg)^2
\end{multline*} 
where $\Gamma$ is a counterclockwise oriented path enclosing $[L_-,L_+]$ but not enclosing 0 such that $f$ is analytic on a neighborhood of the domain bounded by $\Gamma$.
\end{theorem}
Since the eigenvalues of $X^T\Sigma X$ differ from the eigenvalues of $\Sigma^{1/2}XX^T\Sigma^{1/2}$ by $|M-N|$
 zeros, Theorem \ref{thm:main_thm} immediately imply the central limit theorem of the linear statistics for eigenvalues of $\Sigma^{1/2}XX^T\Sigma^{1/2}$.
 
The next proposition provides a criteria of \eqref{basic_assumption}.
  \begin{proposition} \label{prop_criteria}
 	If $m_{fc}(L_+):=\int\frac{d\mu_{fc}(t)}{t-L_+}\ne-1$ or $\gamma_0\ne \big(\int\big(\frac{t}{1-t}\big)^2d\nu(t)\big)^{-1}$, then
 \begin{align}\label{eq155}
 \P(|L_+-\hat L_+|<\epsilon)\to1\quad\forall\epsilon>0.
 \end{align}
 	If $m_{fc}(L_-):=\int\frac{d\mu_{fc}(t)}{t-L_-}\ne-l^{-1}$ or $\gamma_0\ne \big(\int\big(\frac{t}{t-l}\big)^2d\nu(t)\big)^{-1}$, then
 \begin{align}\label{eq156} 
 \P(|L_--\hat L_-|<\epsilon)\to1\quad\forall\epsilon>0.
 \end{align}
 Obviously, \eqref{eq155} and \eqref{eq156} yield \eqref{basic_assumption}.
 \end{proposition} 
Proposition \ref{prop_criteria} is proved in Section \ref{proof_of_examples}. Now we give two explicit examples satisfying \eqref{basic_assumption}.

\begin{example} Suppose
$\nu_1$ is a probability measure supported on $[l,1]$ such that its density function is bounded above and below:
\begin{align} \label{eq157}
	0<\inf_{x\in[l,1]}\frac{d\nu_1(x)}{dx}\le \sup_{x\in[l,1]}\frac{d\nu_1(x)}{dx}<\infty.
\end{align} 
\eqref{eq157} implies that $m_{fc}(L_+)=-\infty$ and $m_{fc}(L_-)=\infty$, so by Proposition \ref{prop_criteria}, \eqref{basic_assumption} must be true when $\nu=\nu_1$. 
\end{example}

\begin{example}\label{example:jacobi_measure}Suppose
 $\nu_2$ is the Jacobi measure on $[l,1]$:
$$\frac{d\nu_2}{dt}=\frac{1}{Z}\cdot d(x)\cdot(1-t)^b$$
where
\begin{itemize}
	\item   $b>1$ is a constant
	\item $d(x)\in C^1([l,1])$ and $d(x)>0$ on $[l,1]$
	\item $Z$ is the normalization constant: $Z=\int_l^1d(t)(1-t)^bdt$.
\end{itemize} 
If
\begin{align}\label{eq158}
\gamma_0\in(0,(\int\frac{t^2}{(1-t)^2}d\nu_2(t))^{-1})\quad\text{and}\quad \gamma_0\not\in\{1,(\int\frac{t^2}{(t-l)^2}d\nu_2(t))^{-1}\}
\end{align} 
then by Proposition \ref{prop_criteria}, \eqref{basic_assumption} is true. Moreover, according to Theorem 2.7 of \cite{Kwak+Lee+Park}, if $\nu=\nu_2$ and \eqref{eq158} holds, then $\rho_{fc}(t)\sim(1-t)^b$ when $t\uparrow L_+$. This means \eqref{basic_assumption} and our main theorem do not require $\rho_{fc}$ to decay like the square root near edges (which is a common requirement for sample covariance matrix).
\end{example}

\begin{definition}
\begin{itemize}
	\item 	Suppose  $\tau\in(0,C_0^{-1})$ where $C_0>1$ is a constant defined in Lemma \ref{auxiliary1}.  Let
\begin{align}\label{eq64}
	D_\tau=\Big\{z=E+\i\eta\in\C_+\Big|\tau\le E\le \tau^{-1}, N^{-1+\tau}\le\eta\le\tau^{-1}, \text{dist}\big(E,[   \hat L_-,\hat  L_+]\big)\ge\tau\Big\}
\end{align}
and
$$D_\tau'=\Big\{z=E+\i\eta\in\C_+\Big| E\in[\tau,\tau^{-1}], \tau\le\eta\le\tau^{-1}\Big\}$$

\item 
Let $m_N$, $m_{fc}$ and $\hat m_{fc}$ denote the Stieltjes transform of $\frac{1}{N}\sum_{i=1}^N\delta_{\lambda_i}$, $\mu_{fc}$ and $\hat\mu_{fc}$ respectively.  

\item For any $z\in\C\backslash\R$, let
\begin{align}\label{eq160}
	\Psi(z)=\frac{1}{N\Im z}+\frac{1}{\sqrt N}.
\end{align}

\end{itemize}
\end{definition}
Notice that $D_\tau$ is a random subset of $\C$ but is independent of $X$. For any   $z\in\C\backslash\R$ we have by \eqref{self_consistent_eq}:
$$m_N(z)=\frac{1}{N}\sum_{i=1}^N\frac{1}{\lambda_i-z}$$
\begin{align}\label{eqn25}
	\frac{1}{m_{fc}(z)}=-z+\gamma_0\int\frac{t}{1+tm_{fc}(z)}dt,\quad\quad\frac{1}{\hat m_{fc}(z)}=-z+\frac{1}{N}\sum_{i=1}^M\frac{\sigma_i}{1+\sigma_i\hat m_{fc}(z)}
\end{align}

An important tool we use to prove the main theorem is the following local law. 
\begin{proposition}\label{thm:strong_local_law}
	Suppose $\tau\in(0,C_0^{-1})$ where $C_0$ is the constant defined in Lemma \ref{auxiliary1}. For any (small) $\epsilon>0$ and (big) $D>0$, if $N>N_0=N_0(\epsilon,D,\tau)$, then 
	\begin{align} 
		\P(|m_N-\hat m_{fc}|\le N^{\epsilon}\Psi^2\quad\forall z\in D_\tau\cup D_\tau')>1-N^{-D}
	\end{align}
\end{proposition}
Proposition \ref{thm:strong_local_law} is proved in Section \ref{sec:strong_local_law}. It is similar to some results in \cite{Knowles+Yin}. See Theorem 3.14--3.16 and Lemma 5.6 of \cite{Knowles+Yin}.  But the results in \cite{Knowles+Yin} have the following two differences from Proposition \ref{thm:strong_local_law}. First, the population matrix in \cite{Knowles+Yin} is deterministic. Because of the randomness of $\Sigma$, we have to define the conditional expectation $\E_k$ (see \ref{conditional_expectation}) in a  way different from that in \cite{Knowles+Yin}. This is to ensure \eqref{eq18} and is also important in the the ``binary tree" argument, see \eqref{eq95} and \eqref{eq159}. 

Second,  for the spectral domain, three subdomains are considered in \cite{Knowles+Yin}: i) the the subdomain in which the real part of the points is far away from $\text{supp}(\hat\rho_{fc})$, ii) the subdomain in which the real part of the points  is near edges of $\text{supp}(\hat\rho_{fc})$, iii) the subdomain in which the real part of the points is in the bulk of $\text{supp}(\hat\rho_{fc})$.   For the second and third cases, some conditions on the regularity of the edges and bulks of $\text{supp}(\hat\rho_{fc})$ are assumed. In this paper, when the real part of the points in the spectral domain is near $\text{supp}(\hat\rho_{fc})$, we do not need the regularity of the edges or bulks (as we saw in Example \ref{example:jacobi_measure}), but we require the imaginary part of the points to be far away from zero (as we saw in the definition of $D_\tau'$).

 It turns out that the differences mentioned above are comparatively minor and we can basically follow the method in  \cite{Knowles+Yin} to prove Proposition \ref{thm:strong_local_law}. For the convenience of the readers we write down all details.

 Proposition \ref{thm:strong_local_law} yields the following corollary.
 \begin{corollary}\label{lemma:extreme_eigenvalues}
 	Suppose $\tau\in(0,C_0^{-1})$ where $C_0$ is the constant defined in Lemma \ref{auxiliary1}. For any (big) $D>0$ we have
 	$$\P\Big(\lambda_1,\ldots,\lambda_{M\wedge N}\text{ are all in }[\hat L_--\tau,\hat L_++\tau]\Big)>1-N^{-D}$$
 	when $N>N_0=N_0(\tau,D)$.
 \end{corollary}
Lemma 10.1 of \cite{Knowles+Yin} proved same conclusion as Corollary \ref{lemma:extreme_eigenvalues} for a model with slight difference. In particular, Lemma 10.1 of \cite{Knowles+Yin} assumes some regularity conditions for the edges and bulks of $\text{supp}(\hat\rho_{fc})$, as we mentioned above. Although our model does not necessarily satisfies the regularity conditions,  the method in the proof of Lemma 10.1 of \cite{Knowles+Yin}  still works. For the convenience of readers we write down the proof of Corollary \ref{lemma:extreme_eigenvalues} in Section \ref{sec:strong_local_law}. 
 
\section{Preliminaries}
\begin{lemma}\label{auxiliary1}
	There exist   $C_0>1$ such that  
	\begin{enumerate}
		\item if $N$ is large enough then
		$$\hat L_+<C_0\quad\quad\text{and}\quad\quad \hat L_->C_0^{-1}$$ 
		\item for any (big) $D>0$, if $N>N_0=N_0(D)$ then
		\begin{align}\label{eq42}
			\P(\lambda_i\in[C_0^{-1},C_0]\quad\forall i\in\{1,\ldots,M\wedge N\})>1-N^{-D}
		\end{align} 
	\end{enumerate}
\end{lemma}
 \begin{proof}
 See Appendix \ref{proofs}.
 \end{proof}

\begin{definition}
For any $z\in\C\backslash\R$, define the $(M+N)\times(M+N)$ matrices:
\begin{align}\label{eq57}
H(z):=	
\begin{pmatrix}
	-\Sigma^{-1}&X\\
X^T&-z
\end{pmatrix},\quad G(z):=H(z)^{-1}
\end{align}
as well as the $M\times M$ matrix:
$$G_M(z):=(-\Sigma^{-1}+z^{-1}XX^T)^{-1}=z\Sigma^{1/2}(\Sigma^{1/2}XX^T\Sigma^{1/2}-z)^{-1}\Sigma^{1/2}$$
and the $N\times N$ matrix:
$$G_N(z):=(X^T\Sigma X-z)^{-1}.$$
\end{definition}

\begin{definition}
	\begin{itemize}
		\item  Suppose $T$ is a subset of $\{1,\ldots,M+N\}$. For $z\in\C\backslash\R$, we use $H^{(T)}(z)$ to denote the $(M+N-\vert T\vert )\times(M+N-\vert T\vert )$ matrix:
		$$(H_{ij}(z))_{i,j\in\{1,\ldots,M+N\}\backslash T}$$
		\item 
		 Let
		$$ G^{(T)}(z)=(H^{(T)}(z))^{-1}$$ 
		\item 
		We also set
		$$\sum_i^{(T)}=\sum_{i:i\not\in T},\quad\sum_{i,j}^{(T)}=\sum_{i:i\not\in T}\sum_{j:j\not\in T}$$
	\end{itemize}
\end{definition}
\begin{remark}
	\begin{enumerate}
		\item When $T=\{i\}$, we use $(i)$ instead of $(\{i\})$ in the above definitions. Similarly, we write $(ij)$ instead of $(\{i,j\})$. We use $(Ti)$ to denote $(T\cup\{i\})$.
		\item In $H^{(T)}(z)$ and $G^{(T)}(z)$ we use the original values of matrix indices. For example, the indices for the rows and columns of $H^{(2)}(z)$ are $1,3,4,\ldots,M+N$.
		\item We often write $H(z)$ as $H$ for convenience. For $G(z)$, $G_M(z)$, $G_N(z)$, $H^{(T)}(z)$ and $G^{(T)}(z)$, we usually omit the variable $z$ too.
	\item It is easy to see that $H^{(T)}(z)$ and $G^{(T)}(z)$ are symmetric matrices.
	\end{enumerate}
\end{remark}

The following lemma is Lemma 4.4 of \cite{Knowles+Yin}.
\begin{lemma}
For any $z\in\C\backslash\R$ we have the following results.
\begin{enumerate}
	\item 
\begin{align}\label{eq26}
G(z)=\begin{pmatrix}
	\Sigma XG_NX^T\Sigma-\Sigma\quad\quad&\Sigma XG_N\\
	G_NX^T\Sigma&G_N
\end{pmatrix}
=\begin{pmatrix}
G_M&z^{-1}G_MX\\
z^{-1}X^TG_M\quad\quad&z^{-2}X^TG_MX-z^{-1}
\end{pmatrix}
\end{align}
\item If $M+1\le p\le M+N$ then
\begin{align}\label{eq15}
\frac{1}{G_{pp}}=-z-\sum_{i,j=1}^MX_{i(p-M)}X_{j(p-M)}G_{ij}^{(p)}.
\end{align}
If $p$ and $q$ are distinct elements in $\{M+1,\ldots,M+N\}$ then
$$G_{pq}=-G_{pp}\sum_{i=1}^MX_{i(p-M)}G_{iq}^{(p)}=-G_{qq}\sum_{i=1}^MG_{pi}^{(q)}X_{i(q-M)}=G_{pp}G_{qq}^{(p)}\sum_{i,j=1}^MX_{i(p-M)}G_{ij}^{(pq)}X_{j(q-M)}$$

\item If $i\in\{1,\ldots,M\}$ then
\begin{align}\label{eq13}
\frac{1}{G_{ii}}=-\frac{1}{\sigma_i}-\sum_{p,q=M+1}^{M+N}X_{i(p-M)}X_{i(q-M)}G_{pq}^{(i)}.
\end{align}
If $i$ and $j$ are distinct elements in $\{1,\ldots,M\}$ then
\begin{align}\label{eq20} 
G_{ij}=-G_{ii}\sum_{p=M+1}^{M+N}X_{i(p-M)}G_{pj}^{(i)}=-G_{jj}\sum_{p=M+1}^{M+N}G_{ip}^{(j)}X_{j(p-M)}=G_{ii}G_{jj}^{(i)}\sum_{p,q=M+1}^{M+N}X_{i(p-M)}X_{j(q-M)}G_{pq}^{(ij)}
\end{align}

\item If $i\in\{1,\ldots,M\}$ and $p\in\{M+1,\ldots,M+N\}$, then
\begin{multline}\label{eq7}
G_{ip}=-G_{pp}\sum_{j=1}^MG_{ij}^{(p)}X_{j(p-M)}=-G_{ii}\sum_{q=M+1}^{M+N}X_{i(q-M)}G_{qp}^{(i)}\\
=G_{ii}G_{pp}^{(i)}\Big(-X_{i(p-M)}+\sum_{j\in\{1,\ldots,M\}\backslash\{i\}\atop q\in\{M+1,\ldots,M+N\}\backslash\{p\}}X_{i(q-M)}X_{j(p-M)}G_{qj}^{(ip)}\Big)\\
=G_{pp}G_{ii}^{(p)}\Big(-X_{i(p-M)}+\sum_{j\in\{1,\ldots,M\}\backslash\{i\}\atop q\in\{M+1,\ldots,M+N\}\backslash\{p\}}X_{i(q-M)}X_{j(p-M)}G_{qj}^{(ip)}\Big)
\end{multline}
 
\item Suppose $r\in\{1,\ldots,M+N\}$ and $s,t$ are elements in $\{1,\ldots,M+N\}\backslash\{r\}$, then
\begin{align}\label{eq9}
G_{st}^{(r)}=G_{st}-\frac{G_{sr}G_{rt}}{G_{rr}}.
\end{align}

\end{enumerate}
\end{lemma}

\begin{definition} For any $k\in\{1,\ldots,M+N\}$ we define the conditional expectation 
	\begin{align}\label{conditional_expectation}
	\E_k[\cdot]=\E[\cdot|\text{ sigma algebra genarated by }\{\text{ entries of }H^{(k)}\}\cup\{\text{ entries of }\Sigma\}].
	\end{align}
	
	For any $z\in\C\backslash\R$ we define
	$$h(z)=-\frac{1}{z}+\frac{1}{N}\sum_{i=1}^M\frac{\sigma_i}{1+z\sigma_i}$$
	$$Z_p=Z_p(z)=(1-\E_p)\sum_{i,j=1}^MX_{i(p-M)}X_{j(p-M)}G_{ij}^{(p)},\quad\forall p\in\{M+1,\ldots,M+N\}$$
	$$Z_i=Z_i(z)=(1-\E_i)\sum_{p,q=M+1}^{M+N}X_{i(p-M)}X_{i(q-M)}G_{pq}^{(i)},\quad\forall i\in\{1,\ldots,M\}$$	
	$$A_i=A_i(z)=\frac{1}{N}\sum_{p=M+1}^{M+N}\frac{(G_{pi})^2}{G_{ii}},\quad\forall i\in\{1,\ldots,M\}$$
	$$A_p=A_p(z)=\frac{1}{N}\sum_{i=1}^{M}\frac{(G_{ip})^2}{G_{pp}},\quad\forall p\in\{M+1,\ldots,M+N\}$$
	$$B_N=B_N(z)=-\frac{1}{N}\sum_{i=1}^M\frac{\sigma_i^2\big[A_i(1+\sigma_im_N)+\sigma_iZ_i(Z_i-A_i)\big]}{(1+\sigma_i m_N)^2(1+\sigma_im_N+\sigma_iZ_i+\sigma_iA_i)}.$$
\end{definition}
\begin{remark}
Because of the randomness of $\Sigma$, the $\E_k[\cdot]$ and $Z_k$ are defined in a slightly different way as those in \cite{Knowles+Yin}. See (5.6), (5.7) and (5.8) of \cite{Knowles+Yin}.
\end{remark}
\begin{lemma}\label{lemma:preliminary}
	For any $z\in\C\backslash\R$, we have the following results.
	\begin{enumerate}
		\item If ${i,j}\subset\{1,\ldots,M\}$ and $\{p,q\}\subset\{M+1,\ldots,M+N\}$ then
	\begin{align}\label{eq21}
	|G_{ij}(z)|=|(G_M)_{ij}(z)|\le\frac{|z|}{|\Im z|},\quad \quad |G_{pq}(z)|=|(G_N)_{(p-M)(q-M)}(z)|\le\frac{1}{|\Im z|},
	\end{align}
\begin{align}\label{eq97}
|G_{ip}(z)|\le \frac{N}{|\Im z|}\max_{a,b}|X_{ab}|
\end{align}
Similarly, if $T\subset\{1,\ldots,M+N\}$, $\{i',j'\}\subset\{1,\ldots,M\}\backslash T$ and $\{p',q'\}\subset\{M+1,\ldots,M+N\}\backslash T$, then	 
\begin{align}\label{eq22} 
 	|G_{i'j'}^{(T)}(z)| \le\frac{|z|}{|\Im z|},\quad |G_{p'q'}^{(T)}(z)| \le\frac{1}{|\Im z|}\quad\text{and }|G_{i'p'}^{(T)}|\le \frac{N}{|\Im z|}\max_{a'\in\{1,\ldots,M\}\backslash T\atop b'\in\{M+1,\ldots,M+N\}\backslash T}|X_{a'b'}|.
\end{align}		
\item Suppose $T\subset\{1,\ldots,M+N\}$ and $p\in\{M+1,\ldots,M+N\}\backslash T$ then
\begin{align}\label{eq24}
\sum_{q\in\{M+1,\ldots,M+N\}\backslash T}|G_{pq}^{(T)}(z)|^2\le \frac{\Im G_{pp}^{(T)}}{\Im z}
\end{align}

  Suppose $S\subset\{1,\ldots,M+N\}$ and $i\in\{1,\ldots,M\}\backslash S$ then
\begin{align}\label{eq25}
	\sum_{j\in\{1,\ldots,M\}\backslash S}|G_{ij}^{(S)}(z)|^2\le2\frac{\|\tilde X\tilde X^T\|}{\Im z}\Im G^{(S)}_{ii}+2  \le 2\|\tilde X\tilde X^T\|\cdot\frac{|z|}{|\Im z|^2}+2
\end{align}
where  $\|\cdot\|$ denotes the operator norm and $\tilde X$ is a matrix constructed from $X$ by removing the rows with indices in $S\cap\{1,\ldots,M\}$ and removing the columns with indices in $\{p-M|p\in S\cap\{M+1,\ldots,M+N\}\}$.

\item  There exists a constant $C_w>0$ depending only on $\gamma_0$ such that the following holds. For any $k>0$ and (big) $D>0$, if $N>N_0=N_0(D,k)$, then 
\begin{align} \label{eq28}
	\P(\|XX^T\|<C_w)>1-N^{-D},\quad \P(\|\tilde X\tilde X^T\|<C_w)>1-N^{-D}
\end{align}
where $\tilde X$ is constructed from $X$ by removing no more than $k$ columns or rows of $X$.

		\item For any $i\in\{1,\ldots,M\} $ and $p\in\{M+1,\ldots,M+N\}$
		\begin{align} \label{eq12}
			\frac{1}{G_{ii}}=-\frac{1}{\sigma_i}-Z_i-m_N+A_i,\quad\quad \frac{1}{G_{pp}}=-z-Z_p-\frac{1}{N}\sum_{i=1}^MG_{ii}+A_p
		\end{align}
		\begin{align}\label{eq18}
			Z_i=-(1-\E_i)\frac{1}{G_{ii}},\quad\quad Z_p=-(1-\E_p)\frac{1}{G_{pp}}
		\end{align}
		\item 
		\begin{align}\label{eq10}
			m_N(z)\hat m_{fc}(z)\big(h(m_N(z))-z\big)=\alpha_1\big(m_N(z)-\hat m_{fc}(z)\big)^2+\alpha_2\big (m_N(z)-\hat m_{fc}(z)\big)
		\end{align}
		where
		\begin{align}\label{eq11}
			\alpha_1:=-\frac{1}{N}\sum_{i=1}^M\frac{\hat m_{fc}(z)\sigma_i^2}{(1+\sigma_i m_N(z))(1+\hat m_{fc}(z)\sigma_i)^2}\quad\quad\alpha_2:=1-\frac{1}{N}\sum_{i=1}^M\frac{\hat m_{fc}^2(z)\sigma_i^2}{(1+\hat m_{fc}(z)\sigma_i)^2}=\frac{\hat m_{fc}^2(z)}{\hat m_{fc}'(z)}.
		\end{align}
		
		\item 
		\begin{multline} \label{eq14}
	h(m_N)-h(\hat m_{fc})=		h(m_N)-z\\
	=B_N-\frac{1}{N}\sum_{p=M+1}^{M+N}A_p+\frac{1}{N}\sum_{p=M+1}^{M+N}\frac{(m_N-G_{pp})^2}{m_N^2G_{pp}} 
			+\frac{1}{N}\sum_{p=M+1}^{M+N}Z_p+\frac{1}{N}\sum_{i=1}^M\frac{\sigma_i^2}{(1+\sigma_im_N)^2}Z_i
		\end{multline}
		
	\end{enumerate}
\end{lemma}

\section{A weak local law on $D_\tau$}

\begin{definition}
For any  $\tau\in(0,1)$, let
\begin{align}\label{eq64}
D_\tau=\Big\{z=E+\i\eta\in\C_+\Big|\tau\le E\le \tau^{-1}, N^{-1+\tau}\le\eta\le\tau^{-1}, \text{dist}\big(E,[   \hat L_-,\hat  L_+]\big)\ge\tau\Big\}
\end{align}	 
\end{definition} 
Since $[\hat L_-,\hat L_+]$ depends on $\Sigma$,  $D_\tau$ is a random subset of $\C$, but it is independent of $X$.

\begin{theorem}\label{local_law_for_main_model}
  For any $\tau\in(0,1)$, (small) $\epsilon>0$ and (big) $D>0$, if $N>N_0=N_0(\epsilon,D,\tau)$ then
\begin{align}\label{eq58} 
\P\Big( |m_N(z)-\hat m_{fc}(z)|\le\frac{N^\epsilon}{N\Im z} \quad\forall z\in D_\tau\Big)>1-N^{-D}
\end{align}
\begin{align}\label{eq74}
\P\Big(|  G_{ij}(z)|\le N^\epsilon\Big(\frac{1}{N\Im z}+\sqrt{\frac{\Im \hat m_{fc}(z)}{N\Im z  }}\Big),\quad\forall i\ne j\in\{1,\ldots,M+N\}\text{ and }z\in D_\tau\Big)>1-N^{-D}
\end{align}
\begin{align}\label{eq72} 
\P\Big(|  G_{ii}(z)+\frac{\sigma_i}{1+\sigma_i\hat m_{fc}(z)}|\le N^\epsilon\Big(\frac{1}{N\Im z}+\sqrt{\frac{\Im \hat m_{fc}(z)}{N\Im z  }}\Big),\quad\forall i \in\{1,\ldots,M\}\text{ and }z\in D_\tau\Big)>1-N^{-D}
\end{align}
\begin{align}\label{eq85}
\P\Big(|  G_{pp}(z)-\hat m_{fc}(z)|\le N^\epsilon\Big(\frac{1}{N\Im z}+\sqrt{\frac{\Im \hat m_{fc}(z)}{N\Im z  }}\Big),\quad\forall p \in\{M+1,\ldots,M+N\}\text{ and }z\in D_\tau\Big)>1-N^{-D}
\end{align}
\end{theorem}
\begin{proof}
Consider the random variable: 
$$A(X,\Sigma,z):=\mathds1_{z\in D_\tau}\cdot|m_N(z)-\hat m_{fc}(z)|\frac{N\Im z}{N^\epsilon}.$$

Let $\sigma(\Sigma)$ be the sigma algebra generated by entries of $\Sigma$ and let $\omega$ denote a sample point in the probability space. Then
\begin{multline*}
\P\Big(\mathds1_{z\in D_\tau}\cdot|m_N(z)-\hat m_{fc}(z)|\le\frac{N^\epsilon}{N\Im z} \Big)=\P\Big(A(X,\Sigma,z)\le1\Big)=\E[\E[\mathds1_{A(X,\Sigma,z)\le1}|\sigma(\Sigma)]]\\
=\int \E[\mathds1_{A(X,\Sigma,z)\le1}|\sigma(\Sigma)](\omega)d\P(\omega)=\int \E[\mathds1_{A(X,\Sigma(\omega),z)\le1}]d\P(\omega)=\int \P(A(X,\Sigma(\omega),z)\le 1)d\P(\omega)
\end{multline*}
where we used the independence of $X$ and $\Sigma$ in the fourth identity.

 Note that $\Sigma(\omega)$   satisfies the conditions on the population matrix in Theorem 3.16 of \cite{Knowles+Yin}. 
Thus, according to Theorem 3.16 of \cite{Knowles+Yin}, if $N>N_0(\epsilon,D,\tau)$  then
$$\P(A(X,\Sigma(\omega),z)\le 1)\ge 1-N^{-D}$$
and  
$$\P\Big(\mathds1_{z\in D_\tau}\cdot|m_N(z)-\hat m_{fc}(z)|\le\frac{N^\epsilon}{N\Im z} \Big)\ge 1-N^{-D}$$
thus by   a classical ``lattice" argument we have
$$\P\Big(\mathds1_{z\in D_\tau}\cdot|m_N(z)-\hat m_{fc}(z)|\le\frac{N^\epsilon}{N\Im z},\quad \forall z\in\Big\{E+\i\eta\Big||E|\le \tau^{-1}, N^{-1+\tau}\le\eta\le\tau^{-1}\Big\}\Big)\ge 1-N^{-D}$$
which yields \eqref{eq58}. The other conclusions can be proved similarly using results from \cite{Knowles+Yin}.
\end{proof}

\section{A weak local law on $D_\tau'$}

\begin{lemma} 
	Suppose $p$ and $q$ are distinct numbers in $\{M+1,\ldots,M+N\}$. Suppose $i$ and $j$ are distinct numbers in  $\{1,\ldots,M\}$. For any (small) $\epsilon'>0$ and (big) $D'>0$, there exists $N_0=N_0(\epsilon',D')>0$ such that if $N>N_0$ and $z\in{\mathbb C }\backslash{\mathbb R }$ then
\begin{align}\label{eq23} 
{\mathbb P}\Big(\Big\vert\sum_{k,s=1}^MX_{k(p-M)}X_{s(q-M)}G_{ks}^{(pq)}\Big\vert \le N^{\epsilon'}\sqrt{\frac{1}{N^2}\sum_{k,s=1}^M\vert G_{ks}^{(pq)}\vert ^2}\Big)>1-N^{-D'}
\end{align}
	\begin{align} \label{eq31} 
		{\mathbb P}\Big(\Big\vert\sum_{s=1}^MX_{s(p-M)}G_{ks}^{(p)}\Big\vert \le N^{\epsilon'}\sqrt{\frac{1}{N}\sum_{s=1}^M\vert G_{ks}^{(p)}\vert ^2}\Big)>1-N^{-D'},\quad\forall k\in\{1,\ldots,M\}
	\end{align}
	\begin{align} \label{eq33} 
		{\mathbb P}\Big(\Big\vert\sum_{k,s\in[1,M]\atop k\ne s}X_{k(p-M)}X_{s(p-M)}G_{ks}^{(p)}\Big\vert \le N^{\epsilon'}\sqrt{\frac{1}{N^2}\sum_{k,s\in[1,M]\atop k\ne s}\vert G_{ks}^{(p)}\vert ^2}\Big)>1-N^{-D'}
	\end{align}
	\begin{align} \label{eq32}
		{\mathbb P}\Big(\Big\vert\sum_{k=1}^M((X_{k(p-M)})^2-\frac{1}{N})G_{kk}^{(p)}\Big\vert \le N^{\epsilon'}\sqrt{\frac{1}{N^2}\sum_{k=1}^M\vert G_{kk}^{(p)}\vert ^2}\Big)>1-N^{-D'}
	\end{align}
	\begin{align} \label{eq36}
	{\mathbb P}\Big(\Big\vert\sum_{k,s=M+1}^{M+N}X_{i(k-M)}X_{j(s-M)}G_{ks}^{(ij)}\Big\vert \le N^{\epsilon'}\sqrt{\frac{1}{N^2}\sum_{k,s=M+1}^{M+N}\vert G_{ks}^{(ij)}\vert ^2}\Big)>1-N^{-D'}
\end{align}
\begin{align} \label{eq40}
	{\mathbb P}\Big(\Big\vert\sum_{k,s\in[M+1,M+N]\atop k\ne s}X_{i(k-M)}X_{i(s-M)}G_{ks}^{(i)}\Big\vert \le N^{\epsilon'}\sqrt{\frac{1}{N^2}\sum_{k,s\in[M+1,M+N]\atop k\ne s}\vert G_{ks}^{(i)}\vert ^2}\Big)>1-N^{-D'}
\end{align}
\begin{align} \label{eq39} 
	{\mathbb P}\Big(\Big\vert\sum_{k=M+1}^{M+N}((X_{i(k-M)})^2-\frac{1}{N})G_{kk}^{(i)}\Big\vert \le N^{\epsilon'}\sqrt{\frac{1}{N^2}\sum_{k=M+1}^{M+N}\vert G_{kk}^{(i)}\vert ^2}\Big)>1-N^{-D'}
\end{align}
\end{lemma}
\begin{proof} 
	Suppose $\mathcal G_1$ is the sigma algebra generated by entries of $H^{(pq)}$. Then $X_{k(p-M)}$ and $X_{s(q-M)}$ are independent of $\mathcal G_1$. Let
	$$B_{ks}=\frac{G_{ks}^{(pq)}}{\sqrt{\sum_{k,s=1}^M\vert G_{ks}^{(pq)}\vert ^2}}.$$
	
	For any natural number $r$ and any sample point $\omega$ in the probability space, we have
	\begin{multline*}
		\E\Big[\Big\vert \sum_{k,s=1}^MX_{k(p-M)}X_{s(q-M)}B_{ks}\Big\vert ^{2r}\Big\vert \mathcal G_1\Big](\omega)
		=\E\Big[\Big(\sum_{k,s=1}^MX_{k(p-M)}X_{s(q-M)}B_{ks}\Big)^r\Big(\sum_{k,s=1}^MX_{k(p-M)}X_{s(q-M)}\bar B_{ks}\Big)^r\Big\vert \mathcal G_1\Big](\omega)\\
		=\E\Big[\Big(\sum_{k,s=1}^MX_{k(p-M)}X_{s(q-M)}B_{ks}(\omega)\Big)^r\Big(\sum_{k,s=1}^MX_{k(p-M)}X_{s(q-M)}\bar B_{ks}(\omega)\Big)^r\Big]=\E\Big[\Big\vert \sum_{k,s=1}^MX_{k(p-M)}X_{s(q-M)}B_{ks}(\omega)\Big\vert ^{2r}\Big]\\
		\le C\Big(\sum_{k,s=1}^M\Big\vert \frac{B_{ks}(\omega)}{N}\Big\vert ^2\Big)^r=CN^{-2r}
	\end{multline*}
	where $C>0$ depends only on $r$. We used  \eqref{eq8} and the Marcinkiewicz-Zygmund inequality (see, for example, Lemma 9 of \cite{Lee+Li}) in the inequality.
	So,
	\begin{multline*}
		{\mathbb P}\Big(\Big\vert \sum_{k,s=1}^MX_{k(p-M)}X_{s(q-M)}G_{ks}^{(pq)}\Big\vert >N^{\epsilon'}\sqrt{\frac{1}{N^2}\sum_{k,s=1}^M\vert G_{ks}^{(pq)}\vert ^2}\Big)={\mathbb P}\Big(\Big\vert N\sum_{k,s=1}^MX_{k(p-M)}X_{s(q-M)}B_{ks}\Big\vert >N^{\epsilon'}\Big)
		\\\le N^{-2r\epsilon'}\E\Big[\Big\vert N\sum_{k,s=1}^MX_{k(p-M)}X_{s(q-M)}B_{ks}\Big\vert ^{2r}\Big]\le CN^{-2r\epsilon'}.
	\end{multline*}
	Choose $r$ large enough such that $2r\epsilon'>D'$, then we complete the proof of the first conclusion. 
	
	Other conclusions can be proved in the same way by using different variants of the  Marcinkiewicz-Zygmund inequality. These  variants of the  Marcinkiewicz-Zygmund inequality can be found in Lemma 9 of \cite{Lee+Li}. Also see Lemma 13 of \cite{Lee+Li} for similar conclusions.
\end{proof}

In this section we provide the weak local law on $D_\tau'$. The weak local law on $D_\tau'$ is similar as Theorem 3.14 and Theorem 3.15 of \cite{Knowles+Yin}.

\begin{theorem}  \label{entrywise_local_law}
Suppose $\tau\in(0,C_0^{-1})$.	For any (small) $\epsilon'>0$ and (big) $D'>0$, if $N>N_0=N_0(\epsilon',D',\tau)$ and $z\in D_\tau'$ then
	\begin{enumerate}
		\item if $p$, $q$ are both in $\{M+1,\ldots,M+N\}$, then
		\begin{align} \label{eq29}
			\P\Big(\Big|G_{pq}(z)-\delta_{pq}m_N(z)\Big|\le  N^{\epsilon'-\frac{1}{2}}  \Big)>1-N^{-D'}
		\end{align}
		\item if $i\in\{1,\ldots,M\}$ and  $p\in\{M+1,\ldots,M+N\}$, then
		\begin{align} \label{eq35} 
			\P(|G_{ip}(z)|\le  N^{\epsilon'-\frac{1}{2}} )>1-N^{-D'}.
		\end{align} 
		\item if $i$ and $j$ are distinct numbers in $\{1,\ldots,M\}$, then
		\begin{align} \label{eq37}
			\P(|G_{ij}(z)|\le N^{\epsilon'-\frac{1}{2}} )\ge 1-N^{-D'}
		\end{align}
		\item if $i\in\{1,\ldots,M\}$, then
		\begin{align} \label{eq41}
			\P\Big(\Big|	G_{ii}(z)+\frac{\sigma_i}{1+\sigma_i m_N(z)}\Big|\le   N^{\epsilon'-\frac{1}{2}} \Big)>1-N^{-D'}.
		\end{align}
	\end{enumerate}
\end{theorem}
\begin{proof}
	\begin{enumerate}
		\item Suppose $p$ and $q$ are distinct numbers in $\{M+1,\ldots,M+N\}$.   By \eqref{eq20} and \eqref{eq22}  
		\begin{align*}
			|G_{pq}|=|G_{pp}G_{qq}^{(p)}\sum_{i,j=1}^MX_{i(p-M)}X_{j(q-M)}G_{ij}^{(pq)}|
			\le \frac{1}{|\Im z|^2}|\sum_{i,j=1}^MX_{i(p-M)}X_{j(q-M)}G_{ij}^{(pq)}|
		\end{align*}
Note  $|z|$ is bounded and $\Im z\ge\tau$  when $z\in D_\tau'$.	By \eqref{eq23}, \eqref{eq25} and \eqref{eq28} we obtain \eqref{eq29} in the case $p\ne q$.
	 
To consider the case $p=q$, we let $u$ and $v$ be two  indices in $\{M+1,\ldots,M+N\}$, not necessarily distinct.  By \eqref{eq15} and \eqref{eq9},
		\begin{multline*}
			\frac{1}{G_{uu}}=-z-\sum_{i,j=1}^MX_{i(u-M)}X_{j(u-M)}G_{ij}^{(u)}\\
			=-z-\frac{1}{N}\sum_{k=1}^MG^{(u)}_{kk}+\sum_{k=1}^MG_{kk}^{(u)}(\frac{1}{N}-(X_{k(u-M)})^2)-\sum_{i,j\in[1,M]\atop i\ne j}X_{i(u-M)}X_{j(u-M)}G_{ij}^{(u)}\\
			=-z-\frac{1}{N}\sum_{k=1}^MG_{kk}+\frac{1}{N}\sum_{k=1}^M\frac{(G_{ku})^2}{G_{uu}}+\sum_{k=1}^MG_{kk}^{(u)}(\frac{1}{N}-(X_{k(u-M)})^2)-\sum_{i,j\in[1,M]\atop i\ne j}X_{i(u-M)}X_{j(u-M)}G_{ij}^{(u)}
		\end{multline*}
		which implies:
		\begin{multline} \label{eq30}
			G_{uu}-G_{vv}=G_{uu}G_{vv}(\frac{1}{G_{vv}}-\frac{1}{G_{uu}})= \frac{G_{uu}}{N}\sum_{k=1}^M(G_{kv})^2-\frac{G_{vv}}{N}\sum_{k=1}^M(G_{ku})^2 \\
			+G_{uu}G_{vv}\sum_{k=1}^MG_{kk}^{(v)}(\frac{1}{N}-(X_{k(v-M)})^2)
			-G_{uu}G_{vv}\sum_{k=1}^MG_{kk}^{(u)}(\frac{1}{N}-(X_{k(u-M)})^2)\\
			-G_{uu}G_{vv}\sum_{i,j\in[1,M]\atop i\ne j}X_{i(v-M)}X_{j(v-M)}G_{ij}^{(v)}+G_{uu}G_{vv}\sum_{i,j\in[1,M]\atop i\ne j}X_{i(u-M)}X_{j(u-M)}G_{ij}^{(u)}
		\end{multline}
		Next, we estimate each term on the right hand side of \eqref{eq30}.
		
		{\bf Estimation of the first two terms in \eqref{eq30}.} By     \eqref{eq7} and   \eqref{eq22},
		$$\Big|\frac{G_{uu}}{N}\sum_{k=1}^M(G_{kv})^2\Big|\le\frac{1}{N|\Im z|}\sum_{k=1}^M\Big|G_{vv}\sum_{r=1}^MG_{rk}^{(v)}X_{r(v-M)}\Big|^2\le \frac{1}{N|\Im z|^3}\sum_{k=1}^M\Big|\sum_{r=1}^MG_{rk}^{(v)}X_{r(v-M)}\Big|^2$$
	Note that $\Im z\ge \tau$ since $z\in D_\tau'$.	By \eqref{eq31}, \eqref{eq25} and \eqref{eq28},   if $N>N_0(\epsilon',D',\tau)$ then
		$$\P\Big(\Big|\frac{G_{uu}}{N}\sum_{k=1}^M(G_{kv})^2\Big|\le N^{\epsilon'-1}\Big) \ge1-N^{-D'} $$
		Similarly we can obtain same bound for the second term in \eqref{eq30}.
		
		{\bf Estimation of the third term and the fourth term in \eqref{eq30}.} By \eqref{eq22}  and \eqref{eq32},   if $N>N_0(\epsilon',D', \tau)$ then
		$$\P\Big(\Big|\text{sum of the third and fourth terms in \eqref{eq30}}\Big|\le  N^{\epsilon'-\frac{1}{2}}\Big) \ge1-N^{-D'} $$
		
		{\bf Estimation of the last two terms in \eqref{eq30}.}  By \eqref{eq33}, \eqref{eq25},  \eqref{eq28}   and \eqref{eq22}, if $N>N_0(\epsilon',D',\tau)$
		$$\P\Big(\Big|\text{sum of the last two terms in \eqref{eq30}}\Big|\le  N^{\epsilon'-\frac{1}{2}}\Big) \ge1-N^{-D'} $$ 
		
		From the above estimations of the terms in \eqref{eq30}, if $N>N_0(\epsilon',D',\tau)$, then
		\begin{align} \label{eq34}
			\P(|G_{uu}-G_{vv}|\le  N^{\epsilon'-\frac{1}{2}})>1-N^{-D'}	
		\end{align}
		For any $p\in\{M+1,\ldots,M+N\}$,  by \eqref{eq26}:
		\begin{align*}
			|G_{pp}-m_N|=|G_{pp}-\frac{1}{N}\sum_{s=M+1}^{M+N}G_{ss}|\le\frac{1}{N}\sum_{s=M+1}^{M+N}|G_{pp}-G_{ss}|	
		\end{align*}
		which together with \eqref{eq34} prove \eqref{eq29} in the case $p= q$.

		\item   By \eqref{eq7}, 
		$$G_{ip}=-G_{pp}\sum_{k=1}^MX_{k(p-M)}G_{ki}^{(p)}$$
		 which together with \eqref{eq22}, \eqref{eq31}, \eqref{eq25} and \eqref{eq28} complete the proof of \eqref{eq35}.  

		\item By \eqref{eq20},
\begin{align}  
	G_{ij}=G_{ii}G_{jj}^{(i)}\sum_{p,q=M+1}^{M+N}X_{i(p-M)}X_{j(q-M)}G_{pq}^{(ij)}
\end{align}		
Then using \eqref{eq36}, \eqref{eq24} and \eqref{eq22}, we proved \eqref{eq37}.

		\item  By \eqref{eq13} and \eqref{eq9},
		\begin{multline} \label{eq38}
			G_{ii}\Big(\frac{1}{G_{ii}}+\frac{1}{\sigma_i}+m_N\Big)=	G_{ii}\Big(\frac{1}{N}\sum_{p=M+1}^{M+N}G_{pp}-\sum_{p,q=M+1}^{M+N}X_{i(p-M)}X_{i(q-M)}G_{pq}^{(i)}\Big)\\
			=	G_{ii}\Big(\frac{1}{N}\sum_{p=M+1}^{M+N}\Big[G_{pp}-G_{pp}^{(i)}\Big]+\sum_{p=M+1}^{M+N}(\frac{1}{N}-(X_{i(p-M)})^2)G_{pp}^{(i)}-\sum_{p,q\in[M+1,M+N]\atop p\ne q}X_{i(p-M)}X_{i(q-M)}G_{pq}^{(i)}\Big)\\
			=\frac{1}{N}\sum_{p=M+1}^{M+N}(G_{pi})^2+G_{ii}\sum_{p=M+1}^{M+N}(\frac{1}{N}-(X_{i(p-M)})^2)G_{pp}^{(i)}-G_{ii}\sum_{p,q\in[M+1,M+N]\atop p\ne q}X_{i(p-M)}X_{i(q-M)}G_{pq}^{(i)}
		\end{multline}
		For the terms on the right hand  side of \eqref{eq38}, using \eqref{eq35} to estimate the first term, using \eqref{eq39}  and \eqref{eq22} to estimate the second term, using \eqref{eq40}, \eqref{eq22} and \eqref{eq24} to estimate the third term,  we conclude that    if $N>N_0(\epsilon',D',\tau)$, then
		\begin{align}\label{eq43}
			\P\Big(\Big|	G_{ii}\Big(\frac{1}{G_{ii}}+\frac{1}{\sigma_i}+m_N\Big)\Big|\le  N^{\epsilon'-\frac{1}{2}} \Big)>1-N^{-D'}.
		\end{align}
		Notice $|1+\sigma_im_N|\ge l|\Im m_N|=\frac{l}{N}\sum_{i=1}^N\frac{\Im z}{|\lambda_i-z|^2}\ge\frac{l\tau}{\max_i(|z|+|\lambda_i|)^2}$, so
		$$\Big|	G_{ii} +\frac{\sigma_i}{1+\sigma_i m_N }\Big|=\Big|\frac{\sigma_i}{1+\sigma_i m_N }\cdot G_{ii}\Big(\frac{1}{G_{ii}}+\frac{1}{\sigma_i}+m_N\Big)\Big|\le \frac{\max_i(|z|+|\lambda_i|)^2}{l\tau }\Big|G_{ii}\Big(\frac{1}{G_{ii}}+\frac{1}{\sigma_i}+m_N\Big)\Big|.$$
This together with \eqref{eq43} and \eqref{eq42} complete the proof of \eqref{eq41}.
	\end{enumerate}	
\end{proof}

\begin{lemma}\label{lemma:some_estimation}
Suppose $\tau\in(0,C_0^{-1})$.	For any (small) $\epsilon'>0$ and (big) $D'>0$, if $N>N_0=N_0(\epsilon',D',\tau)$   then
\begin{enumerate}
\item $\P(|\frac{1}{G_{kk}}|\le N^{\epsilon'},\quad\forall z\in D_\tau')$ for all $k\in\{1,\ldots,M+N\}$;
\item $\P(|A_k|\le N^{\epsilon'-1},\quad\forall z\in D_\tau')>1-N^{-D'}$ for all $k\in\{1,\ldots,M+N\}$;
	\item $\P(|Z_k|\le N^{\epsilon'-\frac{1}{2}},\quad\forall z\in D_\tau')>1-N^{-D'}$ for all $k\in\{1,\ldots,M+N\}$;
	\item $\P(|B_N|\le N^{\epsilon'-1},\quad\forall z\in D_\tau')>1-N^{-D'}$.
\end{enumerate}
\end{lemma}
\begin{proof}
\begin{enumerate}
	\item Suppose $z\in D_\tau'$, $i\in\{1,\ldots,M\}$ and $p\in\{M+1,\ldots,M+N\}$. By \eqref{eq15} and \eqref{eq13}
\begin{align} \label{eq46}
	\frac{1}{G_{ii}}=-\sum_{p'=M+1}^{M+N}(X_{i(p'-M)})^2G_{p'p'}^{(i)}-\sum_{p',q'\in[M+1,M+N]\atop p'\ne q'}X_{i(p'-M)}X_{i(q'-M)}G_{p'q'}^{(i)}-\frac{1}{\sigma_i}
\end{align}
\begin{align}\label{eq45}
\frac{1}{G_{pp}}=-\sum_{i'=1}^M(X_{i'(p-M)})^2G_{i'i'}^{(p)}-\sum_{i',j'\in[1,M]\atop i'\ne j'}X_{i'(p-M)}X_{j'(p-M)}G_{i'j'}^{(p)}-z
\end{align}
Using \eqref{eq44} and \eqref{eq22} to estimate the first term of the RHS of \eqref{eq45} and the first term of the RHS of \eqref{eq46}, using \eqref{eq33} and \eqref{eq22} the second term of the RHS of \eqref{eq45}, using \eqref{eq40} and \eqref{eq22} the second term of the RHS of \eqref{eq46}, we have that if $N>N_0(\epsilon',D',\tau)$ then
\begin{align}\label{eq50}
\P(|\frac{1}{G_{kk}}|\le N^{\epsilon'}\quad\forall k\in\{1,\ldots,M+N\})>1-N^{-D'}.
\end{align}
Using a classic ``lattice" argument we obtain the first statement.
\item 
The first conclusion together with \eqref{eq29}, \eqref{eq35} and \eqref{eq37} and a ``lattice " argument prove the second conclusion.

\item  Suppose $z\in D_\tau'$, $i\in\{1,\ldots,M\}$ and $p\in\{M+1,\ldots,M+N\}$. By definition,

\begin{align}\label{eq47}
Z_p=\sum_{i'=1}^M\Big((X_{i'(p-M)})^2-\frac{1}{N}\Big)G_{i'i'}^{(p)}+\sum_{i',j'\in[1,M]\atop i'\ne j'}X_{i'(p-M)}X_{j'(p-M)}G_{i'j'}^{(p)},
\end{align}
\begin{align}\label{eq48}
	Z_i=\sum_{p'=M+1}^{M+N}\Big((X_{i(p'-M)})^2-\frac{1}{N}\Big)G_{p'p'}^{(i)}+\sum_{p',q'\in[M+1,M+N]\atop p'\ne q'}X_{i(p'-M)}X_{i(q'-M)}G_{p'q'}^{(i)}.
\end{align}
Using \eqref{eq32} and \eqref{eq22} to estimate the first term of the RHS of \eqref{eq47}, using \eqref{eq33}, \eqref{eq25} and \eqref{eq28} to estimate the second term of the RHS of \eqref{eq47}, using \eqref{eq39} and \eqref{eq22} to estimate the first term of the RHS of \eqref{eq48}, using \eqref{eq40}, \eqref{eq24} and \eqref{eq22} the second term of the RHS of \eqref{eq48}, we have for $N>N_0(\epsilon',D',\tau)$ and $k\in\{1,\ldots,M+N\}$:
 $$\P(|Z_k|\le N^{\epsilon'-\frac{1}{2}})>1-N^{-D'}$$
Using a   ``lattice" argument we obtain the third statement.

\item Suppose $z\in D_\tau'$.	Note $|m_N|\le \frac{1}{\Im z}\le\tau^{-1}$ and $\Im m_N=\frac{1}{N}\sum_{i=1}^N\frac{\Im z}{|\lambda_i-z|^2}\ge\frac{\tau}{\max_i(|z|+|\lambda_i|)^2}$. So we have from \eqref{eq42} and the first two conclusions of this lemma that if $N>N_0(\epsilon',D',\tau)$ and $1\le i\le M$ then
\begin{align}\label{eq49}
\P\Big(|1+\sigma_i m_N|\ge l|\Im m_N|\ge  \frac{l\tau}{(|z|+C_0)^2}\text{ and }|1+\sigma_i m_N+\sigma_iZ_i+\sigma_iA_i|\ge \frac{l\tau}{2(|z|+C_0)^2}\Big)>1-N^{-D'}
\end{align}
where $C_0>0$ is defined in Lemma \ref{auxiliary1}. By \eqref{eq49}, the definition of $B_N$ and the first two conclusions of this lemma, we have for $N>N_0(\epsilon',D',\tau)$: 
$$\P(|B_N|\le N^{\epsilon'-1})>1-N^{-D'}$$
Using a   ``lattice" argument we obtain the last statement.
\end{enumerate}
\end{proof}

\begin{theorem}\label{weak_local_law_on_D_tau'}
Suppose $\tau\in(0,C_0^{-1})$.	For any (small) $\epsilon'>0$ and (big) $D'>0$, if $N>N_0=N_0(\epsilon',D',\tau)$  then
$$\P\Big(|m_N-\hat m_{fc}|\le N^{\epsilon'-\frac{1}{2}}\quad\forall z\in D_\tau'\Big)>1-N^{-D'}$$

\end{theorem}

\begin{proof}
  By \eqref{eq14},  Lemma \ref{self_consistent_eq}, Theorem \ref{entrywise_local_law}  and \eqref{eq49}, if $N>N_0(\epsilon',D',\tau) $ then
\begin{align}\label{eq51}
\P \Big(|h(m_N)-z|\le N^{\epsilon'-\frac{1}{2}},\forall z\in D_\tau'\Big)>1-N^{-D'}.
\end{align}
According to \eqref{eq51} and \eqref{eq42}, if $N>N_0(\epsilon',D',\tau)$ then
\begin{align}\label{eq55} 
\P(E_N)>1-N^{-D'}
\end{align}
where
$$E_N=\big\{ |h(m_N)-z|\le N^{\epsilon'-\frac{1}{2}},\forall z\in D_\tau'\big\}\cap\{\lambda_i\in[0,C_0],\forall 1\le i\le N\}.$$

By \eqref{eq10} we have the quadratic equation:
$$\beta_1\cdot\big(m_N(z)-\hat m_{fc}(z)\big)^2+\big (m_N(z)-\hat m_{fc}(z)\big)=\beta_2 $$
where 
$$\beta_1=-\frac{1}{N}\sum_{i=1}^M\frac{\hat m_{fc}'\sigma_i^2}{\hat m_{fc}(1+\sigma_i m_N(z))(1+\hat m_{fc}(z)\sigma_i)^2}$$
$$\beta_2=\frac{m_N\hat m_{fc}'}{\hat m_{fc}}(h(m_N)-z)$$
Suppose $E_N$ holds and $z\in D_\tau'$. By Lemma \ref{auxiliary1}
 $$\Im \hat m_{fc}(z)=\Im z\int\frac{d\hat\mu_{fc}(t)}{|t-z|^2}\ge \frac{\Im z}{(|z|+\hat L_+)^2}\ge \frac{\tau}{(|z|+C_0)^2}$$
 $$\Im m_N(z)=\frac{\Im z}{N}\sum_{i=1}^N\frac{1}{|\lambda_i-z|^2}\ge\frac{\tau}{(|z|+C_0)^2}$$
and therefore
\begin{align}\label{eq54}
|\beta_1|\le C_1,\quad\quad |\beta_2|\le C_1N^{\epsilon'-\frac{1}{2}}
\end{align}
where $C_1>0$ is a constant determined by $\tau$.

According tot he above quadratic equation, we have
$$m_N-\hat m_{fc} =m_{(1)}:=\frac{-1+x_1}{2\beta_1}\quad\text{or}\quad m_N-\hat m_{fc} =m_{(2)}:=\frac{-1+x_2}{2\beta_1}$$
where $x_{1,2}$ are square roots of $1+4\beta_1\beta_2$ such that $x_1$ is close to $-1$ and $x_2=-x_1$ is close to $1$.  So if $N>N_0(\epsilon',D',\tau)$ then we have:
$$|1+x_1|\le |4\beta_1\beta_2|\le 4C_1^2N^{\epsilon'-\frac{1}{2}},\quad|1-x_2|\le|4\beta_1\beta_2|\le 4C_1^2N^{\epsilon'-\frac{1}{2}},$$
$$|m_{(1)}-m_{(2)}|=\Big|\frac{-1+x_1}{2\beta_1}-\frac{-1+x_2}{2\beta_1}\Big|=|\frac{x_2}{\beta_1}|\ge \frac{1}{2|\beta_1|}\ge\frac{1}{2C_1}$$
$$|m_{(2)}|=|\frac{-1+x_2}{2\beta_1}|=|\frac{2\beta_2}{x_2+1}|\le 4C_1N^{\epsilon'-\frac{1}{2}}\quad(\text{since }x_2^2=1+4\beta_1\beta_2)$$ $$|m_{(1)}|\ge|m_{(1)}-m_{(2)}|-|m_{(2)}|\ge \frac{1}{2C_1}-4C_1N^{\epsilon'-\frac{1}{2}}\ge\frac{1}{4C_1}$$
so by continuity, we have
\begin{align}\label{eq52}
m_N-\hat m_{fc} =m_{(1)}=\frac{-1+x_1}{2\beta_1}\quad \forall z\in D_\tau'
\end{align}
or
\begin{align}\label{eq53}
m_N-\hat m_{fc} =m_{(2)}=\frac{-1+x_2}{2\beta_1}\quad \forall z\in D_\tau'
\end{align}

By  definition we have $D_\tau\cap D_\tau'\ne\emptyset$, so Theorem \ref{local_law_for_main_model} implies that \eqref{eq52} cannot happen.   In summary, if $N>N_0(\epsilon',D',\tau)$ and $E_N$ holds, then \eqref{eq53} holds and
$$|m_N-\hat m_{fc}| =|m_{(2)}|\le 4C_1N^{\epsilon'-\frac{1}{2}},\quad\forall z\in D_\tau'.$$

Note that $\epsilon'$ can be arbitrarily small.  So we use \eqref{eq55} and complete the proof.
\end{proof}

\section{A strong local law on $D_\tau\cup D_\tau'$: proof of Proposition \ref{thm:strong_local_law}}\label{sec:strong_local_law}

According to Theorem \ref{local_law_for_main_model}, Theorem \ref{weak_local_law_on_D_tau'} and the fact that $\frac{\Im\hat m_{fc}}{\Im z}\sim1$ on $D_\tau\cup D_\tau'$ (which will be proved in Lemma \ref{lemma:basic}), for any (small) $\epsilon>0$ and  (big) $D>0$, if $N>N_0=N_0(\epsilon,D,\tau)$, then 
\begin{align}\label{eq83}
\P(|m_N-\hat m_{fc}|\le N^{\epsilon}\Psi(z)\quad\forall z\in D_\tau\cup D_\tau')>1-N^{-D}
\end{align}
where
\begin{align}\label{eq75} 
	\Psi(z)=\frac{1}{N\Im z}+\frac{1}{\sqrt N}
\end{align}
as we defined in \eqref{eq160}. In this section we prove that the $\Psi(z)$ in \eqref{eq83} can be improved to $\Psi^2(z)$.

\begin{lemma}\label{lemma:basic}
	Suppose $\tau\in(0,C_0^{-1})$.
There exists a constant $C_\tau>1$ such that if $z$ is in
\begin{align}\label{eq119}
	\{x+\i y|0\le y\le \tau^{-1}, x\in[\tau,\hat L_--\tau]\cup[\hat L_++\tau,\tau^{-1}]\}\cup D_\tau'
\end{align}
 then
\begin{enumerate}
	\item  
\begin{align}\label{eq60}
C_\tau^{-1}\le\frac{\Im \hat m_{fc}(z)}{\Im z}\le C_\tau
\end{align} 
\begin{align}\label{eq61}
 |\hat m_{fc}(z)|\le C_\tau
\end{align}
\item for any $i\in\{1,\ldots,M\}$ we have 
\begin{align}\label{eq70}
C_\tau^{-1}\le|1+\sigma_i\hat m_{fc}(z)|\le C_\tau 
\end{align}
\item for any (big) $D>0$, if $N>N_0=N_0(D,\tau)$ then
\begin{align}\label{eq71}
\P(|G_{ii}|\le C_\tau,\quad\forall i\in\{1,\ldots,M+N\}\text{ and }z\in D_\tau\cup D_\tau' )>1-N^{-D}
\end{align}
\end{enumerate}
\end{lemma}
\begin{proof}
\begin{enumerate}
	\item  
Notice that
\begin{align}\label{eq59}
\frac{\Im \hat m_{fc}(z)}{\Im z}=\int\frac{1}{|t-z|^2}d\hat \mu_{fc}(t)
\end{align}
Since  $z\in \eqref{eq119}$, we have $|t-z|\ge\tau$ for $t$ in the support of $\hat\mu_{fc}$, so by \eqref{eq59} we know the right inequality in \eqref{eq60} is correct. On the other hand, $|t-z|\le |t|+|z|$ is bounded by a constant for all $t$ in the support of $\hat\mu_{fc}$, so by \eqref{eq59} we know the left inequality in \eqref{eq60} is correct. \eqref{eq61} can be proved by noticing
$$|\hat m_{fc}(z)|\le \int\frac{1}{|z-t|}d\hat\mu_{fc}(t)$$
and the fact that $|t-z|\ge\tau$ for $t$ in the support of $\hat\mu_{fc}$.

\item

 According to Lemma 2.5 of \cite{Knowles+Yin}, 
$$x_1=\hat m_{fc}(\hat L_+):=\lim_{\eta\to0+}\hat m_{fc}(\hat L_++\i\eta)$$
where $x_1$ is the unique critical point of the function
\begin{align}\label{eq62}
r\mapsto -\frac{1}{r}+\frac{1}{N}\sum_{j=1}^M\frac{\sigma_j}{1+r\sigma_j}
\end{align}
on $(-(\max \sigma_j)^{-1},0)$. It's easy to see $\hat m_{fc}$ is negative and is increasing on $[\hat L_+,\infty)$, so  
\begin{align}\label{eq63}
|\hat m_{fc}(L_++\tau)-\hat m_{fc}(\hat L_+)|\ge\frac{\tau}{(\hat L_++\tau)^2}
\end{align}

which implies
$$|\hat m_{fc}(L_++\tau)|\le |x_1|-\frac{\tau}{(\hat L_++\tau)^2}<\frac{1}{\max\sigma_j}-\frac{\tau}{(\hat L_++\tau)^2}.$$
Since $\hat m_{fc}(\cdot)$ is increasing on $(\hat L_+,\infty)$, we have for all $x\ge \hat L_++\tau$:
\begin{align} \label{eq65}
1+\sigma_i\hat m_{fc}(x)\ge	1+\sigma_i\hat m_{fc}(L_++\tau)>1-\sigma_i\Big(\frac{1}{\max\sigma_j}-\frac{\tau}{(\hat L_++\tau)^2}\Big)>\frac{l\tau}{(\hat L_++\tau)^2}
\end{align}

If $\gamma_0>1$ then $\hat\mu_{fc}$ does not have an atom at 0, so $\hat m_{fc}>0$ on $(-\infty,\hat L_-)$ and we have
\begin{align} \label{eq66}
	|1+\sigma_i\hat m_{fc}(x)|>1,\quad\forall x\in(-\infty,\hat L_-)
\end{align}

Suppose $\gamma\in(0,1)$. According to Lemma 2.5 of \cite{Knowles+Yin}, 
$$x_1'=\hat m_{fc}(\hat L_-):=\lim_{\eta\to0+}\hat m_{fc}(\hat L_-+\i\eta)$$
where $x_1'$ is the unique critical point of \eqref{eq62} on  $(-\infty,-\frac{1}{\min\sigma_j})$. It's easy to see $\hat m_{fc}$ is  is increasing on $(0,\hat L_-]$, so 
$$|\hat m_{fc}(L_--\tau)-\hat m_{fc}(\hat L_-)|>\frac{\tau}{(\hat L_++\tau)^2}$$
which implies
$$\hat m_{fc}(L_--\tau )<x_1'-\frac{\tau}{(\hat L_++\tau)^2}<-\frac{1}{\min \sigma_j}-\frac{\tau}{(\hat L_++\tau)^2}.$$
Since $\hat m_{fc}(\cdot)$ is increasing on $(0,\hat L_-)$, we have for all $x\in(0,\hat L_-)$:
\begin{align} \label{eq67} 
1+\sigma_i\hat m_{fc}(x)\le	1+\sigma_i\hat m_{fc}(L_--\tau)<-\frac{l\tau}{(\hat L_++\tau)^2}.
\end{align}

Then, by \eqref{eq65}, \eqref{eq66} and \eqref{eq67}, if $x\in[\tau,\hat L_--\tau]\cup[\hat L_++\tau,\tau^{-1}]$, then
$$|1+\sigma_i\hat m_{fc}(x)|\ge C_1:= 1\wedge \frac{l\tau}{(\hat L_++\tau)^2}$$
and for $0<y<C_1\tau^2/2$ we have
$$|\hat m_{fc}(x+\i y)-\hat m_{fc}(x)|=\Big|\int\frac{\i y}{(t-x)(t-x-\i y)}d\hat \mu_{fc}(t)\Big|\le \frac{y}{\tau^2}$$
thus
\begin{align}\label{eq68}
|1+\sigma_i\hat m_{fc}(x+\i y)|\ge|1+\sigma_i\hat m_{fc}(x)|-\frac{y}{\tau^2}>C_1/2.
\end{align}

On the other hand, if $z\in \eqref{eq119}$ and $\Im z\ge C_1\tau^2/2$, then  
\begin{align}\label{eq69}
|1+\sigma_i\hat m_{fc}(z)|\ge l\cdot\Im \hat m_{fc}(z)= l\cdot\Im z\int\frac{d\hat\mu_{fc}(t)}{|t-z|^2}\ge  l\cdot\Im z\cdot\frac{1}{(|z|+\hat L_+)^2}\ge\frac{lC_1\tau^2/2}{(\sqrt2\tau^{-1}+C_0)^2}
\end{align}
where we used the facts $|z|\le\sqrt2\tau^{-1}$ and $\hat L_+\le C_0$ (see Lemma \ref{auxiliary1}) in the last inequality. \eqref{eq68} and \eqref{eq69} complete the proof of the first inequality in \eqref{eq70}. The second inequality in \eqref{eq70} is directly from \eqref{eq61}.
\item Notice that $D_\tau\cup D_\tau'$ is contained in \eqref{eq119}. So \eqref{eq71} comes from Theorem \ref{local_law_for_main_model}, Theorem \ref{entrywise_local_law}, Theorem \ref{weak_local_law_on_D_tau'} and the first two conclusions of this lemma.
\end{enumerate}
\end{proof}

\begin{lemma} \label{lemma:technical}
	Suppose $\tau\in(0,C_0^{-1})$.	For any (small) $\epsilon'>0$ and (big) $D'>0$, if $N>N_0=N_0(\epsilon',D',\tau)$   then
	\begin{enumerate}
		\item  
		\begin{align}\label{eq77} 
\P(|\frac{1}{G_{kk}}|\le N^{\epsilon'},\quad\forall z\in D_\tau\cup D_\tau')	\quad\text{for all }		k\in\{1,\ldots,M+N\},
		\end{align}
\begin{align}\label{eq84}
\P(|m_N|\ge N^{-\epsilon'},\quad\forall z\in D_\tau\cup D_\tau' )>1-N^{-D'}
\end{align}	
	\begin{align}\label{eq78}
	\P(|G_{sk}^{(r)}|\le N^{\epsilon'}\Psi(z), \quad\forall z\in D_\tau\cup D_\tau')>1-N^{-D'}\,\text{for all }k\ne s\in\{1,\ldots,M+N\}\backslash\{r\}
\end{align}
		\begin{align}\label{eq76}
		\P(|G_{ss}^{(r)}|\le 2C_\tau, \quad\forall z\in D_\tau\cup D_\tau')>1-N^{-D'}\quad\text{for all }s\in\{1,\ldots,M+N\}\backslash\{r\}
		\end{align}
where $C_\tau$ was defined in Lemma \ref{lemma:basic};
		\item $\P(|A_k|\le N^{\epsilon'}\Psi^2(z),\quad\forall z\in D_\tau\cup D_\tau')>1-N^{-D'}$ for all $k\in\{1,\ldots,M+N\}$;
		\item $\P(|Z_k|\le N^{\epsilon'}\Psi(z),\quad\forall z\in D_\tau\cup D_\tau')>1-N^{-D'}$ for all $k\in\{1,\ldots,M+N\}$;
		\item $\P(|B_N|\le N^{\epsilon'}\Psi^2(z),\quad\forall z\in D_\tau\cup D_\tau')>1-N^{-D'}$
	\end{enumerate}
where
$\Psi(z)=\frac{1}{N\Im z}+\frac{1}{\sqrt N}$ (as defined in \eqref{eq75}).
\end{lemma}
\begin{proof}
	
Thanks to Lemma \ref{lemma:some_estimation}, it suffices to prove each conclusion with $D_\tau\cup D_\tau'$ replaced by $D_\tau$.
\begin{enumerate}
	\item By \eqref{eq72}, \eqref{eq60}  and \eqref{eq61} we obtain \eqref{eq77} for $1\le k\le M$. Now suppose $M+1\le p\le M+N$ and $z$ is in
\begin{align}\label{eq120}
 \Big\{x+\i y\Big|\tau\le x\le \tau^{-1}, N^{-1+\tau}\le y\le\tau^{-1}\Big\}.
\end{align}
	  By \eqref{eq15} and \eqref{eq9} we have: 
\begin{align} 
	\frac{1}{G_{pp}}=-z-\sum_{i=1}^M(X_{i(p-M)})^2\Big[G_{ii}-\frac{(G_{ip})^2}{G_{pp}}\Big]-\sum_{i,j\in[1,M]\atop i\ne j}X_{i(p-M)}X_{j(p-M)}G_{ij}^{(p)} 
\end{align}
and
\begin{align}\label{eq73}
\frac{1}{G_{pp}}\Big[1-\sum_{i=1}^M(X_{i(p-M)}G_{ip})^2\Big]=-z-\sum_{i=1}^M(X_{i(p-M)})^2G_{ii}-\sum_{i,j\in[1,M]\atop i\ne j}X_{i(p-M)}X_{j(p-M)}G_{ij}^{(p)} .
\end{align}
By \eqref{eq9} and \eqref{eq33}, if $N>N_0(\epsilon',D',\tau)$, then we have with at least $1-N^{-D'}$ probability that:
\begin{multline*}
\Big|\sum_{i,j\in[1,M]\atop i\ne j}X_{i(p-M)}X_{j(p-M)}G_{ij}^{(p)}\Big|\le N^{\epsilon'}\sqrt{\frac{1}{N^2}\sum_{i,j\in[1,M]\atop i\ne j}|G_{ij}^{(p)}|^2}=N^{\epsilon'-1}\sqrt{\sum_{i,j\in[1,M]\atop i\ne j}\Big|G_{ij}-\frac{G_{ip}G_{jp}}{G_{pp}}\Big|^2}\\
\le\sqrt2 N^{\epsilon'-1}\Big(\sqrt{\sum_{i,j\in[1,M]\atop i\ne j}|G_{ij}|^2 }+\frac{1}{|G_{pp}|}\sqrt{\sum_{i,j\in[1,M]\atop i\ne j}|G_{ip}G_{jp}|^2 }\Big)
\end{multline*}
and by \eqref{eq73}:
\begin{multline*}
\mathds1_{z\in D_\tau}\cdot\frac{1}{|G_{pp}|}\Big[1-|\sum_{i=1}^M(X_{i(p-M)}G_{ip})^2|-\sqrt2 N^{\epsilon'-1}\sqrt{\sum_{i,j\in[1,M]\atop i\ne j}|G_{ip}G_{jp}|^2 }\Big]\\
\le \mathds1_{z\in D_\tau}\Bigg(|z|+|\sum_{i=1}^M(X_{i(p-M)})^2G_{ii}|+\sqrt2 N^{\epsilon'-1}\sqrt{\sum_{i,j\in[1,M]\atop i\ne j}|G_{ij}|^2 }\Bigg)
\end{multline*}
where 
\begin{itemize}
	\item $\mathds1_{z\in D_\tau}|\sum_{i=1}^M(X_{i(p-M)}G_{ip})^2|\le N^{\epsilon'-2\tau}$ by \eqref{eq44}, \eqref{eq74} and \eqref{eq73}
	\item  $\mathds1_{z\in D_\tau}\sum_{i,j\in[1,M]\atop i\ne j}|G_{ip}G_{jp}|^2 \le N^{2+\epsilon'-2\tau}$ and $\mathds1_{z\in D_\tau}\sum_{i,j\in[1,M]\atop i\ne j}|G_{ij}|^2 \le N^{2+\epsilon'-2\tau}$ by \eqref{eq74}
	\item  $\mathds1_{z\in D_\tau}|\sum_{i=1}^M(X_{i(p-M)})^2G_{ii}|\le N^{\epsilon'} $ by \eqref{eq44} and \eqref{eq71}.
\end{itemize}
Since $\epsilon'$ can be arbitrarily small, we have for $N>N_0(\epsilon',D',\tau)$:
$$\P(\mathds1_{z\in D_\tau}\frac{1}{|G_{pp}|}\le N^{\epsilon'})>1-N^{-D'}.$$
Using a ``lattice" argument on\eqref{eq120} we prove \eqref{eq77} for $k\in\{M+1,\ldots,M+N\}$, thus complete the proof of \eqref{eq77}.  \eqref{eq84} is directly from \eqref{eq77}, \eqref{eq58} and \eqref{eq85}. \eqref{eq76} comes directly from \eqref{eq71}, \eqref{eq74}, \eqref{eq77} and \eqref{eq9}. \eqref{eq78} is from \eqref{eq74}, \eqref{eq77} and \eqref{eq9}.

\item The second conclusion is directly from \eqref{eq74} and the first conclusion.
 
\item  Suppose  $i\in\{1,\ldots,M\}$, $p\in\{M+1,\ldots,M+N\}$ and $z$ is in \eqref{eq120}. By definition,

\begin{align} \label{eq79}
	Z_p=\sum_{i'=1}^M\Big((X_{i'(p-M)})^2-\frac{1}{N}\Big)G_{i'i'}^{(p)}+\sum_{i',j'\in[1,M]\atop i'\ne j'}X_{i'(p-M)}X_{j'(p-M)}G_{i'j'}^{(p)},
\end{align}
\begin{align} \label{eq80}
	Z_i=\sum_{p'=M+1}^{M+N}\Big((X_{i(p'-M)})^2-\frac{1}{N}\Big)G_{p'p'}^{(i)}+\sum_{p',q'\in[M+1,M+N]\atop p'\ne q'}X_{i(p'-M)}X_{i(q'-M)}G_{p'q'}^{(i)}.
\end{align}
Using \eqref{eq32} and \eqref{eq76} to estimate the first term of the RHS of \eqref{eq79}, using \eqref{eq33} and \eqref{eq78} to estimate the second term of the RHS of \eqref{eq79}, using \eqref{eq39} and \eqref{eq76} to estimate the first term of the RHS of \eqref{eq80}, using \eqref{eq40} and \eqref{eq78} to estimate the second term of the RHS of \eqref{eq80}, we have for $N>N_0(\epsilon',D',\tau)$ and $k\in\{1,\ldots,M+N\}$:
$$\P(\mathds1_{z\in D_\tau}\cdot |Z_k|\le N^{\epsilon'}\Psi(z))>1-N^{-D'}$$
Using a   ``lattice" argument we obtain the third statement.

\item According to \eqref{eq58}, \eqref{eq70} and the second and third conclusions of this lemma,   if $N>N_0(\epsilon',D',\tau)$, then we have with probability at least $1-N^{-D'}$ that:
\begin{align}\label{eq81} 
|1+\sigma_i m_N(z)|\ge \frac{1}{2C_\tau},\quad\forall z\in D_\tau,
\end{align}
\begin{align}\label{eq82} 
|1+\sigma_i m_N(z)+\sigma_i Z_i+\sigma_i A_i|\ge \frac{1}{2C_\tau},\quad\forall z\in D_\tau.
\end{align} 
\eqref{eq81}, \eqref{eq82} and the second and third conclusions of this lemma complete the proof of the fourth conclusion.
\end{enumerate}
\end{proof} 
 
\begin{lemma}\label{lemma:for_binary_tree}
	Suppose $\tau\in(0,C_0^{-1})$ and $p_0\in\{0,1,2,\ldots\}$. For any (small) $\epsilon>0$ and (big) $D>0$, if $N>N_0=N_0(\epsilon,D,\tau,p_0)$, then we have
\begin{align}\label{eq86}
\P(|G_{ij}^{(T)}|\le N^\epsilon\Psi(z),\quad \forall z\in D_\tau\cup D_\tau')>1-N^{-D}
\end{align}
\begin{align}  \label{eq89} 
	\P(|G_{ii}^{(T)}|\le (p_0+1)C_\tau,\quad\forall i\in\{1,\ldots,M+N\}\text{ and }z\in D_\tau\cup D_\tau' )>1-N^{-D}
\end{align}
\begin{align}\label{eq87}
\P\Big(\frac{1}{|G_{ii}^{(T)}| }\le N^\epsilon,\quad \forall z\in D_\tau\cup D_\tau'\Big) >1-N^{-D}
\end{align}
\begin{align}\label{eq88} 
\P\Big(\Big|(1-\E_i)\frac{1}{G_{ii}^{(T)} }\Big|\le N^\epsilon\Psi(z),\quad \forall z\in D_\tau\cup D_\tau'\Big) >1-N^{-D}
\end{align}
for any $i\ne j\in\{1,\ldots,M+N\}\backslash T$ and $T\subset\{1,\ldots,M+N\}$ satisfying $|T|\le p_0$.	
\end{lemma} 
\begin{proof}
We first prove \eqref{eq86}, \eqref{eq89} and \eqref{eq87}. Suppose $|T|=0$. In this case,   Theorem \ref{local_law_for_main_model} and Theorem \ref{entrywise_local_law} imply \eqref{eq86};   \eqref{eq77} gives \eqref{eq87}; \eqref{eq71} gives \eqref{eq89}. If $|T|=1$, then using \eqref{eq9} and the conclusion for $|T|=0$, we obtain \eqref{eq86}, \eqref{eq89} and \eqref{eq87}. Repeating this procedure for finite many times we proved \eqref{eq86}, \eqref{eq89} and \eqref{eq87} for all $|T|\le p_0$.

Now we prove \eqref{eq88}.  Suppose $i\in\{1,\ldots,M\}\backslash T$, $p\in[M+1,M+N]\backslash T$ and $z$ is in \eqref{eq120}. Using \eqref{eq15} and \eqref{eq13} (with $G$ replaced by $G^{(T)}$) and the definition of $\E_k$ we get:
\begin{align}\label{eq90}
(1-\E_i)\frac{1}{G_{ii}^{(T)}}=-\sum_{p',q'\in[M+1,M+N]\backslash T\atop p'\ne q'}X_{i(p'-M)}X_{i(q'-M)}G_{p'q'}^{(iT)}-\sum_{p'\in[M+1,M+N]\backslash T}\Big((X_{i(p'-M)})^2-\frac{1}{N}\Big)G_{p'p'}^{(iT)}
\end{align}
\begin{align}\label{eq91} 
(1-\E_p)\frac{1}{G_{pp}^{(T)}}=-\sum_{i',j'\in[1,M]\backslash T\atop i'\ne j'}X_{i'(p-M)}X_{j'(p-M)}G_{i'j'}^{(pT)}-\sum_{i'\in[1,M]\backslash T}\Big((X_{i'(p-M)})^2-\frac{1}{N}\Big)G_{i'i'}^{(pT)}.
\end{align}
Using \eqref{eq40} (with $G$ replaced by $G^{(T)}$) and \eqref{eq86} to estimate the first term on RHS of \eqref{eq90}, using \eqref{eq39} (with $G$ replaced by $G^{(T)}$) and \eqref{eq89} to estimate the second term on RHS of \eqref{eq90}, using \eqref{eq33} (with $G$ replaced by $G^{(T)}$) and \eqref{eq86} to estimate the first term on RHS of \eqref{eq91}, using \eqref{eq32} (with $G$ replaced by $G^{(T)}$) and \eqref{eq89} to estimate the second term on RHS of \eqref{eq91}, we have for $N>N_0(\epsilon,D,\tau,p_0)$:
$$\P\Big(\mathds1_{z\in D_\tau\cup D_\tau'}\Big|(1-\E_i)\frac{1}{G_{ii}^{(T)}}\Big|\le N^\epsilon\Psi(z)\Big)>1-N^{-D}$$
$$\P\Big(\mathds1_{z\in D_\tau\cup D_\tau'}\Big|(1-\E_p)\frac{1}{G_{pp}^{(T)}}\Big|\le N^\epsilon\Psi(z)\Big)>1-N^{-D}$$
Using the ``lattice" argument we complete the proof of \eqref{eq88}. 
\end{proof}
 
\begin{lemma}\label{lemma:binary_tree}	Suppose $\tau\in(0,C_0^{-1})$.
For any (small) $\epsilon>0$ and (big) $D>0$, if $N>N_0=N_0(\epsilon,D,\tau)$ then  we have:
$$\P(\frac{1}{N}\Big|\sum_{p=M+1}^{M+N}Z_p\Big|\le N^\epsilon\Psi^2,\quad\forall z\in D_\tau\cup D_\tau')>1-N^{-D}$$
$$\P(\frac{1}{N}\Big|\sum_{i=1}^M\frac{\sigma_i^2}{(1+\sigma_i\hat m_{fc})^2}Z_i\Big|\le N^\epsilon\Psi^2,\quad\forall z\in D_\tau\cup D_\tau')>1-N^{-D}$$
\end{lemma} 
	The proof of Lemma \ref{lemma:binary_tree} follows the standard ``binary tree" argument. This argument was first introduced to study the Wigner matrix.  See   \cite{Benaych-Georges+Knowles}.  Later it was used for sample covariance matrix with deterministic population. See \cite{Bloemendal+Erdos+Knowles+Yau+Yin} and \cite{Knowles+Yin}. In our model, although the population is random,  the ``binary tree" argument still works. The key point is to   define $\E_k $ as \eqref{conditional_expectation}. For the convenience of the readers we write down all details of the binary tree argument under the settings of this paper. See Appendix \ref{sec:binary_tree}. Its proof basically follows the steps in the proof of Proposition 6.1 in \cite{Benaych-Georges+Knowles}.

\begin{proof}[Proof of Proposition \ref{thm:strong_local_law}]
For any $z$ in $\C\backslash \R$, we have
$$\Big|\frac{1}{N}\sum_{i=1}^M\frac{\sigma_i^2}{(1+\sigma_im_N)^2}Z_i-\frac{1}{N}\sum_{i=1}^M\frac{\sigma_i^2}{(1+\sigma_i\hat m_{fc})^2}Z_i\Big|\le \frac{1}{N}\sum_{i=1}^M\Big(\frac{|m_N-\hat m_{fc}||Z_i|}{|1+\sigma_im_N||1+\sigma_i\hat m_{fc}|^2}+\frac{|m_N-\hat m_{fc}||Z_i|}{|1+\sigma_im_N|^2|1+\sigma_i\hat m_{fc}|}\Big)$$
thus by \eqref{eq14},  
\begin{multline}\label{eq105}
|	h(m_N)-z|\le 
|B_N|+\frac{1}{N}\sum_{p=M+1}^{M+N}|A_p|+\frac{1}{N}\sum_{p=M+1}^{M+N}\frac{|m_N-G_{pp}|^2}{|m_N^2G_{pp}|} 
+\Big|\frac{1}{N}\sum_{p=M+1}^{M+N}Z_p\Big|\\
+\Big|\frac{1}{N}\sum_{i=1}^M\frac{\sigma_i^2}{(1+\sigma_i\hat m_{fc})^2}Z_i\Big|+\frac{1}{N}\sum_{i=1}^M\Big(\frac{|m_N-\hat m_{fc}||Z_i|}{|1+\sigma_im_N||1+\sigma_i\hat m_{fc}|^2}+\frac{|m_N-\hat m_{fc}||Z_i|}{|1+\sigma_im_N|^2|1+\sigma_i\hat m_{fc}|}\Big)
\end{multline}
Now we estimate the RHS of \eqref{eq105}. Using Lemma \ref{lemma:technical} to control $|B_N|$, 	$|A_p|$ and $\frac{1}{|m_N^2G_{pp}|}$, using Theorem \ref{local_law_for_main_model} and \eqref{eq29} (and a ``lattice" argument) to control $|m_N-G_{pp}|^2$, using Lemma \ref{lemma:binary_tree} to control $\Big|\frac{1}{N}\sum_{p=M+1}^{M+N}Z_p\Big|$ and $
\Big|\frac{1}{N}\sum_{i=1}^M\frac{\sigma_i^2}{(1+\sigma_i\hat m_{fc})^2}Z_i\Big|$, using \eqref{eq70}, \eqref{eq58}, Theorem \ref{weak_local_law_on_D_tau'} and Lemma \ref{lemma:technical} to control the last term, we have for $N>N_0(\epsilon,D,\tau)$:
\begin{align}\label{eq106} 
\P(|	h(m_N)-z|\le N^\epsilon\Psi^2,\quad\forall 
z\in D_\tau\cup D_\tau')>1-N^{-D}
\end{align}

From \eqref{eq84}, Theorem \ref{local_law_for_main_model} and Theorem \ref{weak_local_law_on_D_tau'} we have for $N>N_0(\epsilon,D,\tau)$:
\begin{align}\label{eq107}
\P(\frac{1}{|\hat m_{fc}|}\le N^\epsilon,\quad\forall 
z\in D_\tau\cup D_\tau')>1-N^{-D}.
\end{align}
Since the distance between $D_\tau\cup D_\tau'$ and the support of $\hat\mu_{fc}$ is $\tau$, we have:
\begin{align}\label{eq109}
|\hat m_{fc}(z)|\le \tau^{-1}\quad\text{and}\quad|\hat m_{fc}'(z)|\le \tau^{-2}\quad\forall z\in D_\tau\cup D_\tau'
\end{align}
thus by \eqref{eq70}, \eqref{eq58} and Theorem \ref{weak_local_law_on_D_tau'}, if $N>N_0(\epsilon,D,\tau)$ then
\begin{align}\label{eq108} 
\P(|m_N(z)|\le \frac{2}{\tau}\text{ and }|1+\sigma_i m_N(z)|\ge\frac{1}{2C_\tau},\forall i\in\{1,\ldots,M\}, z\in D_\tau\cup D_\tau')>1-N^{-D}.
\end{align}
By \eqref{eq106}, \eqref{eq107} and \eqref{eq108}, if $N>N_0(\epsilon,D,\tau)$ then 
	\begin{align} \label{eq112}
		\P(E_N)>1-N^{-D}
	\end{align}
	where
\begin{align*}
E_N=\big\{ |h(m_N)-z|\le N^\epsilon\Psi^2, \frac{1}{|\hat m_{fc}|}\le N^\epsilon,|m_N|\le\frac{2}{\tau},|1+\sigma_i m_N(z)|\ge\frac{1}{2C_\tau}\quad\forall 
z\in D_\tau\cup D_\tau',i\in[1,M]\} 
\end{align*}
	
	By \eqref{eq10} we have the quadratic equation:
	$$\beta_1\cdot\big(m_N(z)-\hat m_{fc}(z)\big)^2+\big (m_N(z)-\hat m_{fc}(z)\big)=\beta_2 $$
	where 
	$$\beta_1=-\frac{1}{N}\sum_{i=1}^M\frac{\hat m_{fc}'\sigma_i^2}{\hat m_{fc}(1+\sigma_i m_N(z))(1+\hat m_{fc}(z)\sigma_i)^2}$$
	$$\beta_2=\frac{m_N\hat m_{fc}'}{\hat m_{fc}}(h(m_N)-z)$$
	Suppose $E_N$ holds and $z\in D_\tau\cup D_\tau'$. We have from \eqref{eq70}, \eqref{eq109} and the definition of $E_N$:
	\begin{align} 
		|\beta_1|\le C_1N^\epsilon,\quad\quad |\beta_2|\le C_1N^{2\epsilon}\Psi^2
	\end{align}
	where $C_1>0$ is a constant determined by $\tau$.

	According tot he above quadratic equation, we have
	$$m_N-\hat m_{fc}=m_{(1)}:=\frac{-1+x_1}{2\beta_1}\quad\text{or}\quad m_N-\hat m_{fc}=m_{(2)} :=\frac{-1+x_2}{2\beta_1}$$
	where $x_{1,2}$ are square roots of $1+4\beta_1\beta_2$ such that $x_1$ is close to $-1$ and $x_2=-x_1$ is close to $1$. So if $N>N_0(\epsilon,D,\tau)$ then
	$$|1+x_1|\le |4\beta_1\beta_2|\le 4C_1^2N^{3\epsilon}\Psi^2,\quad|1-x_2|\le|4\beta_1\beta_2|\le 4C_1^2N^{3\epsilon}\Psi^2$$
$$|m_{(1)}-m_{(2)}|=\Big|\frac{-1+x_1}{2\beta_1}-\frac{-1+x_2}{2\beta_1}\Big|=|\frac{x_2}{\beta_1}|\ge \frac{1}{2|\beta_1|}\ge\frac{1}{2C_1}N^{-\epsilon}$$
$$|m_{(2)}|=|\frac{-1+x_2}{2\beta_1}|=|\frac{2\beta_2}{x_2+1}|\le 4C_1N^{2\epsilon}\Psi^2\quad(\text{since }x_2^2=1+4\beta_1\beta_2)$$ 
\begin{align}\label{eq122}
|m_{(1)}|\ge|m_{(1)}-m_{(2)}|-|m_{(2)}|\ge \frac{1}{2C_1}N^{-\epsilon}-4C_1N^{2\epsilon}\Psi^2\ge\frac{1}{4C_1}N^{-\epsilon}
\end{align}	
	so by continuity, we have
	\begin{align} \label{eq110}
		m_N-\hat m_{fc} =m_{(1)}=\frac{-1+x_1}{2\beta_1}\quad \forall z\in D_\tau\cup D_\tau'
	\end{align}
	or
	\begin{align} \label{eq111}
		m_N-\hat m_{fc} =m_{(2)}=\frac{-1+x_2}{2\beta_1}\quad \forall z\in  D_\tau\cup D_\tau'
	\end{align}
However, Theorem \ref{local_law_for_main_model} and \eqref{eq122} imply that \eqref{eq110} cannot happen.   In summary, if $N>N_0(\epsilon,D,\tau)$ and $E_N$ holds, then \eqref{eq111} holds and
	$$|m_N-\hat m_{fc}| =|m_{(2)}|\le 4C_1N^{2\epsilon}\Psi^2,\quad\forall z\in D_\tau\cup D_\tau'.$$
Note that $\epsilon$ can be arbitrarily small, so we use \eqref{eq112} and complete the proof.
\end{proof}

\begin{proof}[Proof of Corollary \ref{lemma:extreme_eigenvalues}]
Lemma 10.1 of \cite{Knowles+Yin} proved same conclusion for a model with slight difference. But the method also works for our model. For the convenience of readers we write down the details.

Because of \eqref{eq42}, it suffices to show that if $N>N_0(\tau,D)$ then
\begin{align}\label{eq118}
	\P\Big(\text{none of }\lambda_1\ldots,\lambda_N\text{ is in }[\tau,\hat L_--\tau]\cup[\hat L_++\tau,\tau^{-1}]\Big)>1-N^{-D}.
\end{align}
 
 Without loss of generality assume $\tau<\frac{1}{100}$.  
Set $\epsilon=0.01$. If we have $\lambda_i=x_0\in [\tau,\hat L_--\tau]\cup[\hat L_++\tau,\tau^{-1}]$, then for $y_0:=N^{-\frac{1}{2}-\epsilon}$, we have $x_0+\i y_0\in D_\tau$ and:
\begin{align*}
&\Im m_N(x_0+\i y_0)=\frac{1}{N}\sum_{j=1}^N\frac{y_0}{|\lambda_j-x_0-\i y_0|^2}\ge \frac{1}{Ny_0}=N^{-\frac{1}{2}+\epsilon}\\
&\Im\hat m_{fc}(x_0+\i y_0)\ge C_\tau^{-1}y_0=C_\tau^{-1}N^{-\frac{1}{2}-\epsilon}\quad\text{(by \eqref{eq60})}
\end{align*}
thus for $N>N_0(\tau)$:
\begin{align}\label{eq117}
|m_N(x_0+\i y_0)-\hat m_{fc}(x_0+\i y_0)|\ge\frac{1}{2}N^{-\frac{1}{2}+\epsilon}>N^\epsilon(N^{-\frac{1}{2}+\epsilon}+\frac{1}{\sqrt N})^2=N^\epsilon\Psi^2(x_0+\i y_0).
\end{align}
However, according to Proposition \ref{thm:strong_local_law}, when $N>N_0(\tau,D)$, then with probability no less than $1-N^{-D}$, the inequality \eqref{eq117} does not hold. So \eqref{eq118} is true and the corollary is proved.
\end{proof}

\section{Proof of the main theorem}

In this section we prove Theorem \ref{thm:main_thm}. Suppose its assumptions hold. First of all we suppose
\begin{itemize}
	\item $d\in(0,\frac{L_-}{10}\wedge C_0^{-1})$ is a small constant such that $f$ is analytic on a neighborhood of the closed rectangle with vertices $L_++2d\pm 2d\i$ and $L_--2d\pm 2d\i$;
	\item $\Gamma$ is the boundary of the rectangle described above with counterclockwise orientation;
	\item $$\Omega_N=\big\{|\hat L_--L_-|<\frac{d}{2}\quad\text{and}\quad |\hat L_+-L_+|<\frac{d}{2}\big\}$$ $$\tilde\Omega_N=\big\{\lambda_1,\ldots,\lambda_{M\wedge N}\text{ are all in }[\hat L_--\frac{d}{2},\hat L_++\frac{d}{2}]\big\}\cap\Omega_N$$
\end{itemize}
 By the assumption \eqref{basic_assumption} and Corollary \ref{lemma:extreme_eigenvalues}, we have
 $$\P(\tilde\Omega_N)\to1.$$

Suppose $\tilde\Omega_N$ holds, then we have:
\begin{multline}\label{eqn17}
\frac{1}{\sqrt N}\sum_{i=1}^Nf(\lambda_i)-\sqrt N\int f(t)d\mu_{fc}(t)\\
=\frac{1}{\sqrt N}\Bigg[\frac{1}{2\pi\i}\sum_{i=1}^{M\wedge N}\oint_\Gamma\frac{f(\xi)d\xi}{\xi-\lambda_i}-\frac{N}{2\pi\i}\int\rho_{fc}(t)\oint_\Gamma\frac{f(\xi)d\xi}{\xi-t}dt+f(0)[(N-M)^+-N(1-\gamma_0)^+]\Bigg]\\
=\frac{1}{\sqrt N2\pi\i}\oint_\Gamma f(\xi)\Bigg[N\int\frac{\rho_{fc}(t)}{t-\xi}dt-\sum_{i=1}^{M\wedge N}\frac{1}{\lambda_i-\xi}\Bigg]d\xi+\frac{f(0)}{\sqrt N}[(N-M)^+-N(1-\gamma_0)^+]\\
=\frac{\sqrt N}{2\pi\i}\oint_\Gamma f(\xi)(m_{fc}(\xi)-m_N(\xi))d\xi+\frac{f(0)}{\sqrt N}[(N-M)^+-N(1-\gamma_0)^+]
\end{multline}
where we used  
$$\oint_\Gamma f(\xi) \int\frac{\rho_{fc}(t)}{t-\xi}dtd\xi=\oint_\Gamma f(\xi) \Big(\int\frac{d\mu_{fc}(t)}{t-\xi}+\frac{(1-\gamma_0)^+}{\xi}\Big)d\xi=\oint_\Gamma f(\xi) \int\frac{d\mu_{fc}(t)}{t-\xi}d\xi$$
and
$$\oint_\Gamma f(\xi)\sum_{i=1}^{M\wedge N}\frac{1}{\lambda_i-\xi}d\xi=\oint_\Gamma f(\xi) \Big(\sum_{i=1}^N\frac{1}{\lambda_i-\xi}+\frac{(N-M)^+}{\xi}\Big)d\xi=\oint_\Gamma f(\xi) \sum_{i=1}^N\frac{1}{\lambda_i-\xi}d\xi$$
in the last step of \eqref{eqn17}.

Note that $\P(\tilde\Omega_N)\to1$ and
$$\Big|\frac{f(0)}{\sqrt N}[(N-M)^+-N(1-\gamma_0)^+]\Big|\le f(0)\sqrt N|\gamma_0-\frac{M}{N}|\to0\quad\text{(by \eqref{eqn18})}.$$

So it suffices to prove that
$$\mathds1_{\tilde\Omega_N}\frac{\sqrt N}{2\pi\i}\oint_\Gamma f(\xi)(m_{fc}(\xi)-m_N(\xi))d\xi=\mathds1_{\tilde\Omega_N}\frac{\sqrt N}{2\pi\i}\oint_\Gamma f(\xi)(\hat m_{fc}(\xi)-m_N(\xi))d\xi+\mathds1_{\tilde\Omega_N}\frac{\sqrt N}{2\pi\i}\oint_\Gamma f(\xi)(m_{fc}(\xi)-\hat m_{fc}(\xi))d\xi$$
converges in distribution to the limiting distribution in the statement of Theorem \ref{thm:main_thm}.

When $\tilde\Omega_N$ holds, we have by definition:
$$\text{dist}(\xi,\text{supp}(\hat m_{fc}))\ge d\quad\text{and}\quad \text{dist}(\xi,\lambda_i)\ge d\quad\forall \xi\in \Gamma, i\in\{1,\ldots,N\}$$
and then
$$|m_N(\xi)|\le d^{-1},\quad |\hat m_{fc}(\xi)|\le d^{-1},\quad\forall \xi\in\Gamma.$$

Choose $\tau=0.01\wedge\frac{d}{2}$, so if $\tilde\Omega_N$ holds then $\Gamma\cap\{z|\Im z>N^{-1+\tau}\}\subset D_\tau\cup D_\tau'$.

 By Proposition \ref{thm:strong_local_law} with $\epsilon=0.1$ and $D=1$, if $N>N_0(d)$, then  with probability larger than $1-N^{-1}$ we have
\begin{multline*}
\Big|\mathds1_{\tilde\Omega_N}\frac{\sqrt N}{2\pi\i}\oint_\Gamma(\hat m_{fc}(\xi)-m_N(\xi))f(\xi)d\xi\Big|\\
\le\Big|\mathds1_{\tilde\Omega_N}\frac{\sqrt N}{2\pi\i}\int_{\Gamma\cap\{|\Im\xi|<N^{-1+\tau}\}}(\hat m_{fc}(\xi)-m_N(\xi))f(\xi)d\xi\Big|+\Big|\mathds1_{\tilde\Omega_N}\frac{\sqrt N}{2\pi\i}\int_{\Gamma\cap\{|\Im\xi|>N^{-1+\tau}\}}(\hat m_{fc}(\xi)-m_N(\xi))f(\xi)d\xi\Big| \\
\le \Big(\frac{\sqrt N}{2\pi}4N^{-1+\tau}\cdot\frac{2}{d}+\frac{\sqrt N}{2\pi}\int_{\Gamma\cap\{|\Im\xi|>N^{-1+\tau}\}}2N^{0.1}(\frac{1}{N^2|\Im z|^2}+\frac{1}{N})d\xi\Big)\sup_{\xi\in\Gamma}|f(\xi)|\\
\le\big( \frac{2}{d}N^{-0.5+\tau}+N^{-0.3}\big)\sup_{\xi\in\Gamma}|f(\xi)|
\end{multline*}
So $\mathds1_{\tilde\Omega_N}\frac{\sqrt N}{2\pi\i}\oint_\Gamma(\hat  m_{fc}(\xi)-m_N(\xi))f(\xi)d\xi$ converges in  distribution to 0 and now we only need to show that 
$$\mathds1_{\tilde\Omega_N}\frac{\sqrt N}{2\pi\i}\oint_\Gamma(m_{fc}(\xi)-\hat m_{fc}(\xi))f(\xi)d\xi$$
or 
$$\mathds1_{\Omega_N}\frac{\sqrt N}{2\pi\i}\oint_\Gamma(m_{fc}(\xi)-\hat m_{fc}(\xi))f(\xi)d\xi$$
converges in distribution to the limiting distribution in the statement of Theorem \ref{thm:main_thm}.

\begin{lemma}\label{lemma:bound}
There exists a constant $C_d>0$ depending on $d$ such that
\begin{itemize}
	\item $$|1+tm_{fc}(z)|>C_d \quad\forall z\in\Gamma, t\in[l,1]$$
	\item  if $\Omega_N$
	holds then
	$$|1+\sigma_i\hat m_{fc}(z)|>C_d\quad\forall z\in\Gamma, 1\le i\le M$$
\end{itemize}

\end{lemma}
\begin{proof}
The second conclusion is induced from Lemma \ref{lemma:basic}. To prove the first conclusion, we cite the following result from Page 2271--2272 of \cite{Hachem+Hardy+Najim}:

\vspace{0.2cm}

For any interval $I\subset(-\infty,0)\cup(0,1)\cup(l^{-1},\infty)$, $-m_{fc}(x)$ is decreasing on $I$ and the function
$$r(x):=\frac{1}{x}+\gamma_0\int\frac{t}{1-tx}d\nu(t)$$
is the inverse of $-m_{fc}$ on $-m_{fc}(I)$, thus $r(x)$ is decreasing on $-m_{fc}(I)$.

\vspace{0.2cm}

(Note that the   Stieltjes transform in \cite{Hachem+Hardy+Najim} differs by a sign from the one in this paper.) 
\begin{itemize}
	\item To estimate $|1+ t m_{fc}(L_++2d)|$, we take $I=[L_++2d,\infty)$. From  the above result we see that $(0,-m_{fc}(L_++2d))$ is contained in the domain of $r$, so $|m_{fc}(L_++2d)|<1$ and $|1+t m_{fc}(L_++2d)|$ is larger than a constant which depends only on $d$. 
	\item If $\gamma_0>1$, then $\mu_{fc}$ does not have an atom at origin. So $m_{fc}(L_--2d)>0$ and $|1+t m_{fc}(L_--2d)|>1$.
	\item If $\gamma_0\in(0,1)$, we take $I=(0,L_--2d)$. Then $-m_{fc}(x)$ approaches $\infty$ when $x\to0+$. From the above result we see that $(-m_{fc}(L_--2d),\infty)$ is contained in the domain of $r$, so $m_{fc}(L_--2d)<-l^{-1}$ and $|1+t m_{fc}(L_--2d)|$ is larger than a constant which depends only on $d$.
\end{itemize}
	By continuity there exists $C_1>0$ and $y_0>0$ both determined by $d$ such that if $|y|<y_0$ then
$$|1+tm_{fc}(L_++2d+\i y)|>C_1\quad\text{and}\quad |1+tm_{fc}(L_--2d+\i y)|>C_1.$$
When $z\in\Gamma$ but $|\Im z|\ge y_0$, then 	
$$|1+t m_{fc}(z)|\ge l|\Im m_{fc}(z)|\ge  l y_0\int\frac{d\mu_{fc}(t)}{|z-t|^2}$$
which is bounded below since $|t-z|$ is bounded above uniformly. So the first conclusion is proved.	
\end{proof}

\begin{definition}
For $z\in\C\backslash\R$ and $1\le i\le M$ we define
$$g_i(z)=\frac{\sigma_i}{1+\sigma_im_{fc}(z)},\quad \hat g_i(z)=\frac{\sigma_i}{1+\sigma_i\hat m_{fc}(z)}$$
\end{definition}
 \begin{lemma} 
	Suppose $a_1>0$, $a_2>0$ are constants. Then
	$${\mathbb P}\Big(\Big\vert \frac{1}{N^{\frac{1}{2}+a_1}}\sum_{i=1}^N(g_i(\xi)-\E[g_i(\xi)])\Big\vert \le a_2,\quad\forall\xi\in\Gamma\Big)\to1\quad \text{as } N\to\infty$$
	$${\mathbb P}\Big(\Big\vert \frac{1}{N^{\frac{1}{2}+a_1}}\sum_{i=1}^N(g_i^2(\xi)-\E[g_i^2(\xi)])\Big\vert \le a_2,\quad\forall\xi\in\Gamma\Big)\to1\quad \text{as } N\to\infty$$
\end{lemma}
\begin{proof}
	It can be proved in the same way as Lemma 19 in \cite{Lee+Li}. Note that $g_i$ are bounded on $\Gamma$ thanks to Lemma \ref{lemma:bound}.
\end{proof}
\begin{lemma}\label{loop_equatoin}
	Suppose $z\in\C\backslash\R$. Then we have
	\begin{align}\label{eqn24}
		\hat m_{fc}-m_{fc}=A_N(\hat m_{fc}-m_{fc})^2+B_N
	\end{align}	
	where
	\begin{align*}
		A_N=\frac{m_{fc}'}{m_{fc}}\frac{M}{N}\int\Big(\frac{t}{1+tm_{fc}}\Big)^2d\nu(t)-\frac{m_{fc}'\hat m_{fc}}{N m_{fc}}\sum_{i=1}^M(g_i^2\hat g_i)
	\end{align*}
	\begin{multline*}
		B_N=-\frac{m_{fc}'}{N }\sum_{i=1}^M(g_i-\E g_i)+\frac{m_{fc}'\hat m_{fc}}{N m_{fc}}(\hat m_{fc}-m_{fc})\sum_{i=1}^M(g_i^2-\E[g_i^2])	-\frac{m_{fc}'\hat m_{fc}}{ m_{fc}}(\frac{M}{N}-\gamma_0)\int\frac{td\nu(t)}{1+tm_{fc}}\\
		+m_{fc}'(\hat m_{fc}-m_{fc})\int\frac{t^2d\nu(t)}{(1+tm_{fc})^2}(\frac{M}{N}-\gamma_0) 
		+\frac{m_{fc}'(m_{fc}-\hat m_{fc})}{N m_{fc}}\sum_{i=1}^M (g_i-\E g_i).
	\end{multline*}
\end{lemma}
\begin{proof}
By \eqref{eqn25}, 
\begin{align}\label{eqn27} 
\hat m_{fc}-m_{fc}=\hat m_{fc}m_{fc}(\frac{1}{m_{fc}}-\frac{1}{\hat m_{fc}})=-\hat m_{fc}m_{fc}\Big[\frac{1}{N}\sum_{i=1}^M(\hat g_i-\E g_i)+(\frac{M}{N}-\gamma_0)\int\frac{td\nu(t)}{1+tm_{fc}}\Big]
\end{align}
Note that
\begin{multline}\label{eqn26}
\sum_{i=1}^M(\hat g_i-\E g_i)=\sum_{i=1}^M( g_i-\E g_i)+\sum_{i=1}^M(\hat g_i- g_i)=\sum_{i=1}^M( g_i-\E g_i)+(m_{fc}-\hat m_{fc})\sum_{i=1}^Mg_i\hat g_i\\
=\sum_{i=1}^M( g_i-\E g_i)+(m_{fc}-\hat m_{fc})(\sum_{i=1}^Mg_i^2+(m_{fc}-\hat m_{fc})\sum_{i=1}^Mg_i^2\hat g_i)\\=\sum_{i=1}^M( g_i-\E g_i)+(m_{fc}-\hat m_{fc})\sum_{i=1}^Mg_i^2+(m_{fc}-\hat m_{fc})^2\sum_{i=1}^Mg_i^2\hat g_i\\
=\sum_{i=1}^M( g_i-\E g_i)+(m_{fc}-\hat m_{fc})\sum_{i=1}^M(g_i^2-\E[g_i^2])+(m_{fc}-\hat m_{fc})^2\sum_{i=1}^Mg_i^2\hat g_i-M(\hat m_{fc}-m_{fc})\E[g_i^2].
\end{multline}
Plugging \eqref{eqn26} into \eqref{eqn27}, we have
\begin{multline}\label{eqn28}
\hat m_{fc}-m_{fc}=-\frac{m_{fc}\hat m_{fc}}{N}\sum_{i=1}^M( g_i-\E g_i)+\frac{m_{fc}\hat m_{fc}}{N}(\hat m_{fc}-m_{fc})\sum_{i=1}^M(g_i^2-\E[g_i^2])\\-\frac{m_{fc}\hat m_{fc}}{N}(\hat m_{fc}-m_{fc})^2\sum_{i=1}^Mg_i^2\hat g_i
-\hat m_{fc}m_{fc}(\frac{M}{N}-\gamma_0)\int\frac{td\nu(t)}{1+tm_{fc}}+\frac{M}{N}\hat m_{fc}m_{fc}(\hat m_{fc}-m_{fc})\E[g_i^2]
\end{multline}
In \eqref{eqn28}, we replace the last term by
$$m_{fc}(\hat m_{fc}-m_{fc})\E[g_i^2]\Big(\frac{M}{N}(\hat m_{fc}-m_{fc})+m_{fc}(\frac{M}{N}-\gamma_0)\Big)+(\hat m_{fc}-m_{fc})m_{fc}^2\gamma_0\E[g_i^2],$$
then move the term $(\hat m_{fc}-m_{fc})m_{fc}^2\gamma_0\E[g_i^2]$ to the left hand side, then use the fact that 
\begin{align}\label{eqn29}
	1- m_{fc}^2\gamma_0\E[g_i^2]=1-\gamma_0m_{fc}^2\int \frac{t^2d\nu(t)}{(1+tm_{fc})^2}=\frac{m_{fc}^2}{m_{fc}'}
\end{align}
then we obtain \eqref{eqn24}. Here we used the definition of $g_i$ in the first identity of \eqref{eqn29} and used the self-consistent equation \eqref{self_consistent_eq} to find the $m_{fc}'$ for the second identity of \eqref{eqn29}.
 
\end{proof}
Let
$$\Omega_N'=\Omega_N\cap\Big\{\Big\vert \frac{1}{N^{\frac{1}{2}+0.1c_0}}\sum_{i=1}^N(g_i(\xi)-\E[g_i(\xi)])\Big\vert \le 1,\Big\vert \frac{1}{N^{\frac{1}{2}+0.1c_0}}\sum_{i=1}^N(g_i^2(\xi)-\E[g_i^2(\xi)])\Big\vert \le 1\quad\forall\xi\in\Gamma\Big\} $$
where $c_0\in(0,0.01)$ is defined in Definition \ref{def:model_of_X}.
Then $\P(\Omega')\to1$. On $\Omega'$ we have for all $z\in\Gamma $:
$$|A_N|\le C_1,\quad|B_N|\le C_1 N^{-0.5+0.1c_0} $$ where  $C_1>0$ is a constant determined by $d$. By Lemma \ref{loop_equatoin}, on $\Omega_N'$ we have for $z\in\Gamma$
$$\hat m_{fc}-m_{fc}=\frac{1+x_1}{2A_N}\quad\text{or}\quad \hat m_{fc}-m_{fc}=\frac{1+x_2}{2A_N}$$
where $x_{1,2}$ are square roots of $1-4A_NB_N$ such that $x_1$ is close to $-1$ and $x_2$ is close to $1$. So if $N>N_0(d)$ and $\Omega_N'$ holds, then
$$|1+x_1|\le 4C_1^2N^{-0.5+0.1c_0},\quad|1-x_2|\le4C_1^2 N^{-0.5+0.1c_0},\quad\forall z\in\Gamma$$
$$\Big|\frac{1+x_1}{2A_N}-\frac{1+x_2}{2A_N}\Big|\ge \frac{1}{2C_1},\forall z\in\Gamma$$
so by continuity, we have
\begin{align}\label{eqn23}
	\hat m_{fc}-m_{fc}=\frac{1+x_1}{2A_N},\forall z\in\Gamma
\end{align}
or
\begin{align}\label{eqn22}
\hat m_{fc}-m_{fc}=\frac{1+x_2}{2A_N},\forall z\in\Gamma
\end{align}
Suppose $z_0\in\Gamma$ and $\Im z_0=2d$, by Lemma \ref{lemma:convergence} and Theorem \ref{local_law_for_main_model}, we see $m_{fc}(z_0)-\hat m_{fc}(z_0)=m_{fc}(z_0)-  m_N(z_0)+m_N(z_0)-\hat m_{fc}(z_0)$ converges in probability to 0. Note that if $N>N_0(d)$, then $\quad|\frac{1+x_2}{2A_N(z_0)}|\ge \frac{1}{2C_1}$ holds on $\Omega_N'$, so \eqref{eqn22} cannot happen, so we have \eqref{eqn23}. Plugging \eqref{eqn23} into \eqref{eqn24}, we see that if $N>N_0(d)$, $\Omega_N'$ holds and $z\in\Gamma$ then 
$$(1+x_1)(1-x_1)=4A_NB_N$$
and
\begin{align}\label{main_estimation}
|\hat m_{fc}-m_{fc}|=|\frac{1+x_1}{2A_N}|\le |\frac{2B_N}{1-x_1}|\le |2B_N|\le 2C_1N^{-0.5+0.1c_0}
\end{align}
Putting \eqref{main_estimation} into the definitions of  $B_N$ we see that there exists a constant $C_2=C_2(d)>0$ such that if $\Omega_N'$ holds and $N>N_0(d)$ then 
	\begin{align*}
	B_N=-\frac{m_{fc}'}{N }\sum(g_i-\E g_i)+R_N
\end{align*}
where
$$|R_N|\le C_2N^{-\frac{1}{2}-c_0}$$ 
and therefore by \eqref{eqn24}, \eqref{main_estimation} and estimation of $A_N$,
$$\hat m_{fc}-m_{fc}=-\frac{m_{fc}'}{N }\sum(g_i-\E g_i)+W_N$$
where
$$|W_N|\le 4C_1^3N^{-1+0.2c_0}+C_2N^{-\frac{1}{2}-c_0}.$$

So, for $N>N_0(d)$, 
\begin{multline*}
\Big|\mathds1_{\Omega_N'}\frac{\sqrt N}{2\pi\i}\oint_\Gamma(m_{fc}(\xi)-\hat m_{fc}(\xi))f(\xi)d\xi-\mathds1_{\Omega_N'}\frac{\sqrt N}{2\pi\i}\oint_\Gamma\frac{m_{fc}'}{N }\sum(g_i-\E g_i)f(\xi)d\xi\Big|
\le C_3 N^{-c_0}\sup_{\xi\in\Gamma}|f(\xi)|
\end{multline*}
for some constant $C_3$. Since both $\P(\Omega_N')$ and $\P(\Omega_N)$ go to 1, we see that 
$$\mathds1_{\Omega_N'}\frac{\sqrt N}{2\pi\i}\oint_\Gamma(m_{fc}(\xi)-\hat m_{fc}(\xi))f(\xi)d\xi$$
has the same limit in distribution as
\begin{multline*}
\frac{\sqrt N}{2\pi\i}\oint_\Gamma\frac{m_{fc}'}{N }\sum(g_i-\E g_i)f(\xi)d\xi=\frac{1}{\sqrt N}\sum_{i=1}^M\Bigg\{\frac{1}{2\pi\i}\oint_\Gamma m_{fc}'(\xi)g_i(\xi)f(\xi)d\xi-\E\Bigg[\frac{1}{2\pi\i}\oint_\Gamma m_{fc}'(\xi)g_i(\xi)f(\xi)d\xi\Bigg]\Bigg\}\\
=	\frac{1}{\sqrt N}\sum_{i=1}^M\Bigg\{\frac{1}{2\pi\i}\oint_\Gamma \frac{\sigma_im_{fc}'(\xi)f(\xi)}{1+\sigma_im_{fc(\xi)}}d\xi-\E\Bigg[\frac{1}{2\pi\i}\oint_\Gamma \frac{\sigma_im_{fc}'(\xi)f(\xi)}{1+\sigma_im_{fc(\xi)}}d\xi\Bigg]\Bigg\}
	\end{multline*}
which, by central limit theorem, converges in distribution to a centered Gaussian distribution with variance
\begin{multline*}
\E\Bigg[\Bigg(\frac{1}{2\pi\i}\oint_\Gamma \frac{\sigma_im_{fc}'(\xi)f(\xi)}{1+\sigma_im_{fc}(\xi)}d\xi\Bigg)^2\Bigg]-\Bigg(\E\Bigg[\frac{1}{2\pi\i}\oint_\Gamma \frac{\sigma_im_{fc}'(\xi)f(\xi)}{1+\sigma_im_{fc}(\xi)}d\xi\Bigg]\Bigg)^2\\
=-\frac{1}{4\pi^2}\oint_\Gamma\oint_\Gamma\int_l^1\frac{t^2m_{fc}'(\xi_1)m_{fc}'(\xi_2)f(\xi_1)f(\xi_2)}{(1+tm_{fc}(\xi_1))(1+tm_{fc}(\xi_2))}d\nu(t)d\xi_1d\xi_2
+\frac{1}{4\pi^2\gamma_0^2}\Bigg(\oint_\Gamma f(\xi)m_{fc}'(\xi)(\xi+\frac{1}{m_{fc}(\xi)})d\xi\Bigg)^2
\end{multline*}
where we used the fact $\int_l^1\frac{td\nu(t)}{1+tm_{fc}(\xi)}=\gamma_0^{-1}(\xi+\frac{1}{m_{fc}(\xi)})$ in the last identity. So the proof of the main theorem is complete.

\section{Proof of Proposition \ref{prop_criteria}}\label{proof_of_examples}

\begin{definition}
Define $D=(-\infty,0)\cup(0,1)\cup(\frac{1}{l},\infty)$ and
$$g(x)=\frac{1}{x}+\gamma_0\int\frac{t}{1-xt}d\nu(t)\quad x\in D$$
	$$h_0(x):=\int_l^1\Big(\frac{xt}{1-xt}\Big)^2d\nu(t)\quad x\in D$$
	$$h_1(x):=\frac{1}{M}\sum_{i=1}^M\Big(\frac{x\sigma_i}{1-x\sigma_i}\Big)^2\quad x\in \R\backslash\{0,\sigma_1^{-1},\ldots,\sigma_M^{-1}\}.$$
\end{definition}

According to Remark 2.2 and the first paragraph of
Appendix A of \cite{Hachem+Hardy+Najim}, the support of $\rho_{fc}$ must be the union of a few   finite intervals:
$$\text{supp}(\rho_{fc})=[a_1,b_1]\cup\cdots\cup[a_k,b_k].$$

The following lemma is  Proposition 2.3 of \cite{Hachem+Hardy+Najim}. 
\begin{lemma}\label{prop2.3}
	A number $c$ satisfies $c=b_i$ for some $i\in\{1,\ldots, k\}$ (i.e., $c$ is the right endpoint of a connected components of $\text{supp}(\rho_{fc})$) if and only if one of the following two conditions is satisfied:
	\begin{enumerate}
		\item there exists $x\in D$ such that $c=g(x)$, $g'(x)=0$ and $g''(x)>0$
		\item there exists $x\in\partial D$ and $\epsilon>0$ such that: i) $(x-\epsilon,x)\subset D$, ii) $g(t)$ is decreasing on $(x-\epsilon,x)$ and iii) $c=\lim\limits_{t\to x-}g(t)$ 
	\end{enumerate}
	A number $c'$ satisfies $c'=a_i$ for some $i\in\{1,\ldots, k\}$ (i.e., $c'$ is the left endpoint of a connected components of $\text{supp}(\rho_{fc})$) if and only if one of the following two conditions is satisfied:
	\begin{enumerate}
		\item[3.] there exists $x\in D$ such that $c'=g(x)$, $g'(x)=0$ and $g''(x)<0$
		\item[4.] there exists $x\in\partial D$ and $\epsilon>0$ such that: i) $(x,x+\epsilon)\subset D$, ii) $g(t)$ is decreasing on $(x,x+\epsilon)$ and iii) $c'=\lim\limits_{t\to x+}g(t)$ 
	\end{enumerate}
\end{lemma} 
\begin{remark}\label{remark:2cases}
Proposition 2.3 of \cite{Hachem+Hardy+Najim} requires the edges in this lemma to be ``soft" (i.e., the density function $\rho_{fc}(t)$ goes to a finite value as $t$ approaches this edge). But this condition is automatically satisfied by our model because $\gamma_0\ne0$. It is known that only $a_1$ may fail to be soft and this happen only when $\gamma_0=1$.
\end{remark}

\begin{corollary}\label{coro_[L_-,L_+]}
In our model, $[L_-,L_+]$ satisfy the following conditions.
\begin{itemize}
	\item $L_+=1+\gamma_0\int\frac{t}{1-t}d\nu(t)$ or $L_+=g(x)$ for some $x\in(0,1)$ satisfying $h_0(x)=\frac{1}{\gamma_0}$.
	\item $L_-=1-\gamma_0\int\frac{lt}{t-l}d\nu(t)$ or $L_-=g(x)$ for some $x\in(-\infty,0)\cup(\frac{1}{l},\infty)$ satisfying $h_0(x)=\frac{1}{\gamma_0}$.
\end{itemize}
\end{corollary}
\begin{remark}\label{remark:two_cases_of_gamma_0}
Suppose $L_-=g(x)$ for some $x\in(-\infty,0)\cup(\frac{1}{l},\infty)$. If $\gamma_0>1$, then $x\in(-\infty,0)$. This is because the condition $g'(x)=0$ implies $\int(\frac{tx}{1-tx})^2d\nu(t)=\gamma_0^{-1}<1$, so $x$ cannot be in $(l^{-1},\infty)$, otherwise the integrand in  $\int(\frac{tx}{1-tx})^2d\nu(t)$ is larger than $1$. Similarly, if $\gamma_0\in(0,1)$, then $x\in(l^{-1},\infty)$.
\end{remark}
\begin{proof} Notice that $\partial D=\{0,1,\frac{1}{l}\}$.
For $L_+$, we see that $1$ is the only point in $\partial D$ which may satisfy Condition 2 in Lemma \ref{prop2.3} with a finite $c$. So $L_+$, as a right endpoint of a connected component of $\text{supp}(\rho_{fc})$, may be $\lim_{t\to1-}g(t)=1+\gamma_0\int\frac{t}{1-t}d\nu(t)$. On the other hand, if $x\in D$ satisfies Condition 1 in Lemma \ref{prop2.3}, then we claim that $x$ must be in $(0,1)$. We prove this claim by contradiction.
\begin{itemize}
	\item If $x\in(-\infty,0)$, by the condition $g''(x)>0$,
\begin{align}\label{eqn2}
	\int_l^1\Big(\frac{-tx}{1-xt}\Big)^3d\nu(t)>\frac{1}{\gamma_0}.
\end{align}	
According to the condition $g'(x)=0$, we have $\frac{1}{\gamma_0}=\int_l^1\Big(\frac{tx}{1-xt}\Big)^2d\nu(t)$. Plug this into \eqref{eqn2},  we have
\begin{align}\label{eqnn3}
	\int_l^1\Big(\frac{tx}{1-xt}\Big)^2\Big[\frac{-tx}{1-xt}-1\Big]d\nu(t)>0.
\end{align}
But \eqref{eqnn3} cannot hold since the integrand in it is negative because $x<0$.
\item If $x\in(\frac{1}{l},\infty)$, then similarly,
\begin{align} 
	\int_l^1\Big(\frac{tx}{xt-1}\Big)^3d\nu(t)<\frac{1}{\gamma_0}=	\int_l^1\Big(\frac{tx}{xt-1}\Big)^2d\nu(t).
\end{align}	
So  we have 
\begin{align}\label{eqnn5}
	\int_l^1\Big(\frac{tx}{xt-1}\Big)^2\Big[\frac{tx}{xt-1}-1\Big]d\nu(t)<0.
\end{align}
But \eqref{eqnn5} cannot hold since the integrand in it is positive because $x_0>1/l$. 
\end{itemize}
So the claim is true and we proved the first conclusion.

For $L_+$, we see that $\frac{1}{l}$ is the only point in $\partial D$ which may satisfy Condition 4 in Lemma \ref{prop2.3} with a finite $c$. So $L_-$, as a left endpoint of a connected component of $\text{supp}(\rho_{fc})$, may be $\lim_{t\downarrow l^{-1}}g(t)=1-\gamma_0\int\frac{lt}{t-l}d\nu(t)$. On the other hand, if $x\in D$ satisfies Condition 3 in Lemma \ref{prop2.3}, then we claim that $x$ must be in $(-\infty,0)\cup(l,\infty)$. We prove this claim by contradiction. Suppose $x\in(0,1)$ satisfies Condition 3 in Lemma \ref{prop2.3}. Then by $g''(x)<0$ we have 
$$	\int_l^1\Big(\frac{tx}{xt-1}\Big)^3d\nu(t)>\frac{1}{\gamma_0}.$$
But $\frac{1}{\gamma_0}$ equals $	\int_l^1\Big(\frac{tx}{xt-1}\Big)^2d\nu(t)$, so
 \begin{align} 
 	\int_l^1\Big(\frac{tx}{xt-1}\Big)^2\Big[\frac{tx}{xt-1}-1\Big]d\nu(t)>0
 \end{align}
which cannot be true because $x<1$. So the claim that 
 $x\in(-\infty,0)\cup(l,\infty)$ is true and we proved the second conclusion.
\end{proof}

\begin{lemma}\label{lemma:weaker_condition}
If \begin{align} \label{eq130}
	L_+=g(x_0) 
\end{align}
for some $x_0\in(0,1)$ with $h_0(x_0)=\frac{1}{\gamma_0}$, then 
$$\P(|L_+-\hat L_+|<\epsilon)\to1\quad\forall\epsilon>0.$$
If \begin{align}  \label{eq147} 
  L_-=g(x_0')
\end{align}
for some   $x_0'\in(-\infty,0)\cup(l^{-1},\infty)$ with $h_0(x_0')=\frac{1}{\gamma_0}$, then 
$$\P(|L_--\hat L_-|<\epsilon)\to1\quad\forall\epsilon>0.$$
\end{lemma}
\begin{proof}
	By Lemma 2.4--2.6 of \cite{Knowles+Yin}, 
\begin{align}\label{eq131}
\hat L_+=\frac{1}{x_1}+\frac{1}{N}\sum_{i=1}^M\frac{\sigma_i}{1-x_1\sigma_i}
\end{align}
where $x_1\in (0,\frac{1}{\max \sigma_i})$ with $h_1(x_1)=\frac{N}{M}$.
	Note that $x_0$ is deterministic and is independent of $N$, but $x_1$ is random and depends on $N$.

	Suppose $c>0$ is a small enough constant such that $[x_0-c,x_0+c]\subset(0,1)$. We claim that
	\begin{align}\label{eqn11}
		\P(x_1\in[x_0-c,x_0+c])\to1.
	\end{align}
	In fact, since $h_0$ is increasing on $[0,1]$, there exists $a>0$ such that
	$$h_0(x_0-c)+a<h_0(x_0)<h_0(x_0+c)-a.$$
	Note that both $\sqrt M(h_0(x_0-c)-h_1(x_0-c))$ and $\sqrt M(h_0(x_0+c)-h_1(x_0+c))$ converge in distribution to  Gaussian distributions, so
	\begin{align}\label{eqn10}
		\P\Big(|h_0(x_0\pm c)-h_1(x_0\pm c)|<a\Big)\to1.
	\end{align}
	When $|h_0(x_0\pm c)-h_1(x_0\pm c)|<a$ holds, we have
	$$h_1(x_0-c)<h_0(x_0-c)+a<h_0(x_0)<h_0(x_0+c)-a<h_1(x_0+c)$$
Notice $h_1(x_1)-h_0(x_0)\to0$ since  $|\frac{1}{\gamma_0}-\frac{N}{M}|\to0$. So if $N$ is large enough and $|h_0(x_0\pm c)-h_1(x_0\pm c)|<a$ holds,
	$$h_1(x_0-c)<h_1(x_1)<h_1(x_0+c)$$ 
	thus by monotonicity of $h_1$:
	\begin{align}\label{eqn9}
		x_0-c\le x_1\le x_0+c\quad\text{whenever }|h_0(x_0\pm c)-h_1(x_0\pm c)|<a
	\end{align}
	\eqref{eqn10} and \eqref{eqn9} prove \eqref{eqn11}.
	
	Now suppose $\epsilon>0$. Choose $\delta$ small enough such that $[x_0-\delta,x_0+\delta]\in(0,1)$
	and 
	\begin{align}\label{eqn14}
		\Big|\int_l^1\frac{t}{1-x_0t}d\nu(t)-\int_l^1\frac{t}{1-(x_0\pm\delta)t}d\nu(t)\Big|<\epsilon
	\end{align}
	and
	$$|\frac{1}{x_0}-\frac{1}{x_0\pm\delta}|<\epsilon.$$
	Since $\frac{1}{\sqrt M}\sum_{i=1}^M\frac{\sigma_i}{1-(x_0\pm\delta)\sigma_i}-\sqrt M\int_l^1\frac{t}{1-(x_0\pm\delta)t}d\nu(t)$ converges in distribution to a normal distribution,  we have
	\begin{align}\label{eqn12}
		\P\Big(\Big|\frac{1}{ M}\sum_{i=1}^M\frac{\sigma_i}{1-(x_0\pm\delta)\sigma_i}-\int_l^1\frac{t}{1-(x_0\pm\delta)t}d\nu(t)\Big|<\epsilon\Big)\to1.
	\end{align}
	By \eqref{eqn11} and \eqref{eqn12}, the event
	\begin{align}\label{eqn13}
		\{|x_0-x_1|<\delta\}\cap\Big\{\Big|\frac{1}{ M}\sum_{i=1}^M\frac{\sigma_i}{1-(x_0\pm\delta)\sigma_i}-\int_l^1\frac{t}{1-(x_0\pm\delta)t}d\nu(t)\Big|<\epsilon\Big\}
	\end{align}
	has a probability going to 1. On \eqref{eqn13},  we have by \eqref{eqn14}:
	$$ \int\frac{t}{1-x_0t}d\nu(t)-2\epsilon<\frac{1}{M}\sum_{i=1}^M\frac{\sigma_i}{1-x_1\sigma_i}<\int\frac{t}{1-x_0t}d\nu(t)+2\epsilon$$
	and then for large enough $N$
	\begin{multline*}
		|L_+-\hat L_+|\le |\frac{1}{x_1}-\frac{1}{x_0}|+\Big|\frac{1}{N}\sum_{i=1}^M\frac{\sigma_i}{1-x_1\sigma_i}-\gamma_0\int\frac{t}{1-x_0t}d\nu(t)\Big|\quad\text{(by \eqref{eq130} and \eqref{eq131})} \\
		\le|\frac{1}{x_0+\delta}-\frac{1}{x_0}|\vee|\frac{1}{x_0-\delta}-\frac{1}{x_0}|+ \frac{M}{N}\Big|\frac{1}{M}\sum_{i=1}^M\frac{\sigma_i}{1-x_1\sigma_i}-\int\frac{t}{1-x_0t}d\nu(t)\Big|
		+|\frac{M}{N}-\gamma_0|\int\frac{t}{1-x_0t}d\nu(t)\\
		\le\epsilon+\frac{M}{N}\cdot2\epsilon+\epsilon\quad\text{(since $|\frac{M}{N}-\gamma_0|$ goes to 0)}
	\end{multline*}
	Recall $\P(\eqref{eqn13})\to1$, so  we have proved
	$$\P(|L_+-\hat L_+|\le 2\epsilon(1+\frac{M}{N}))\to1.$$
Since $\frac{M}{N}$ is bounded and $\epsilon$ can be arbitrarily small, the first conclusion of Lemma \ref{lemma:weaker_condition} is proved.
	
Now we estimate $|L_--\hat L_-|$. Suppose \eqref{eq147} holds for some   $x_0'\in(-\infty,0)\cup(l^{-1},\infty)$ with $h_0(x_0')=\frac{1}{\gamma_0}$. By Remark \ref{remark:two_cases_of_gamma_0}, 	$x_0'\in(-\infty,0)$ if $\gamma_0>1$; $x_0'\in(l^{-1},\infty)$ if $\gamma_0\in(0,1)$.
	
By Lemma 2.4--2.6 of \cite{Knowles+Yin}   we have the following.   $\hat L_-=\frac{1}{x_1'}+\frac{1}{N}\sum_{i=1}^M\frac{\sigma_i}{1-x_1'\sigma_i}$
		where 
		$x_1'$ is a point in $(-\infty,0)\cup(\frac{1}{\min_i\sigma_i},\infty)$ such that $h_1(x_1')=\frac{N}{M}$. Similarly as what we explained in Remark \ref{remark:two_cases_of_gamma_0}, if $\gamma_0>1$, then $x_1'\in(-\infty,0)$; if $\gamma_0\in(0,1)$, then $x_1'\in(\frac{1}{\min_i\sigma_i},\infty)$.
	 
	Then using similar argument as above we can prove \eqref{eqn11} with $x_0$ and $x_1$ replaced by $x_0'$ and $x_1'$ respectively. Then similarly as above we can prove 
	$$\P(|L_--\hat L_-|\le 2\epsilon(1+\frac{M}{N}))\to1.$$

	Since $\frac{M}{N}$ is bounded and $\epsilon$ can be arbitrarily small, the second conclusion of Lemma \ref{lemma:weaker_condition} is proved. So we have proved Lemma \ref{lemma:weaker_condition}.
\end{proof}

\begin{proof}[Proof of Proposition \ref{prop_criteria}] 

Suppose $m_{fc}(L_+)\ne-1$. If $L_+= 1+\gamma_0\int\frac{t}{1-t}d\nu(t)$, then by Lemma \ref{prop2.3}, $$m_{fc}(L_+)=\lim_{t\uparrow1}m_{fc}(\frac{1}{t}+\gamma_0\int\frac{s}{1-ts}d\nu(s)).$$
By Proposition 2.1 of \cite{Hachem+Hardy+Najim}, 
$$\lim_{t\uparrow1}m_{fc}(\frac{1}{t}+\gamma_0\int\frac{s}{1-ts}d\nu(s))=-\lim_{t\uparrow1}t=-1$$
which is contradictory to the assumption $m_{fc}(L_+)\ne-1$. Notice that the Stieltjes transform in \cite{Hachem+Hardy+Najim} differs by a sign from the Stieltjes in this paper. So we proved $L_+\ne 1+\gamma_0\int\frac{t}{1-t}d\nu(t)$. By Corollary \ref{coro_[L_-,L_+]}, $L_+=g(x)$ for some $x\in(0,1)$ satisfying $h_0(x)=\frac{1}{\gamma_0}$. Then using Lemma \ref{lemma:weaker_condition} we obtain \eqref{eq155} under the assumption $m_{fc}(L_+)\ne-1$.

Now we prove \eqref{eq155} under the assumption that $\gamma_0\ne \big(\int\big(\frac{t}{1-t}\big)^2d\nu(t)\big)^{-1}$. By Corollary \ref{coro_[L_-,L_+]} and Lemma \ref{lemma:weaker_condition}, without loss of generality, we can suppose  
\begin{align}\label{eq149}
L_+=1+\gamma_0\int\frac{t}{1-t}d\nu(t)
\end{align}
By the second condition in Lemma \ref{prop2.3}, $g(t)$ is decreasing when $t$ is close enough but smaller than 1. So 
$$\lim_{t\uparrow1}\big(-\frac{1}{x^2}+\gamma_0\int(\frac{t}{1-xt})^2d\nu(t)\big)=\lim_{t\uparrow1}g'(t)\le0$$
which implies
$$\int(\frac{t}{1-t})^2d\nu(t)\le\frac{1}{\gamma_0}.$$
 This together with the assumption that $\gamma_0\ne (\int(\frac{t}{1-t})^2d\nu(t))^{-1}$ yields:
	$$\int(\frac{t}{1-t})^2d\nu(t)<\frac{1}{\gamma_0}.$$
	
As we saw in \eqref{eq131},
\begin{align}\label{eqs139}
	\hat L_+=\frac{1}{x_1}+\frac{1}{N}\sum_{i=1}^M\frac{\sigma_i}{1-x_1\sigma_i}
\end{align}
where $x_1\in (0,\frac{1}{\max \sigma_i})$ satisfying $h_1(x_1)=\frac{N}{M}$.

Since $\frac{1}{\sqrt M}\sum(\frac{\sigma_i}{1-\sigma_i})^2-\sqrt M\int(\frac{t}{1-t})^2d\nu(t)$ converges in distribution to a Gaussian distribution, we have
\begin{align}\label{eqs142} 
	\P(Q_N)\to1
\end{align}
where
$$Q_N=\Big\{\Big|\frac{1}{  M}\sum(\frac{\sigma_i}{1-\sigma_i})^2- \int(\frac{t}{1-t})^2d\nu(t)\Big|\le \frac{1}{2}\Big(\frac{1}{\gamma_0}-\int(\frac{t}{1-t})^2d\nu(t)\Big)\Big\}.$$
When $N$ is large enough and $Q_N$ holds, we must have $x_1>1$, because otherwise 
$$h_1(x_1)\le h_1(1)\le \int(\frac{t}{1-t})^2d\nu(t)+\frac{1}{2}\Big(\frac{1}{\gamma_0}-\int(\frac{t}{1-t})^2d\nu(t)\Big)=\frac{1}{2}\Big(\frac{1}{\gamma_0}+\int(\frac{t}{1-t})^2d\nu(t)\Big)<\frac{1}{\gamma_0}$$
which is contradictory to $h_1(x_1)=\frac{N}{M}\to\frac{1}{\gamma_0}$. Moreover, $x_1>1$ implies $x_1\in(1,\frac{1}{\max\sigma_i})$ and then
\begin{multline}\label{eqs143}
	\Big|\frac{1}{  M}\sum\frac{\sigma_i}{1-\sigma_i}-\frac{1}{  M}\sum\frac{\sigma_i}{1-x_1\sigma_i}\Big|\le|x_1-1|\frac{1}{M}\sum\frac{\sigma_i^2}{(1-\sigma_i)(1-x_1\sigma_i)}\\
	\le|1-x_1|\frac{1}{M}\sum\frac{\sigma_i^2}{(1-x_1\sigma_i)^2}=|1-x_1|\cdot\frac{N}{Mx_1^2}\quad\text{(by definition of $x_1$)}\\
	\le \frac{1-\max\sigma_i}{\max\sigma_i}\cdot\frac{N}{M}
\end{multline}

Suppose $\epsilon>0$. Since $\text{supp}\nu=[l,1]$, we have
$$\P(\max\sigma_i\in(1-\epsilon,1))\to1.$$ 
Similarly as \eqref{eqs142} we have
$$\P(Q_N')\to1$$
where
$$Q_N'=\Big\{\Big|\frac{1}{  M}\sum\frac{\sigma_i}{1-\sigma_i}- \int\frac{t}{1-t}d\nu(t)\Big|\le\epsilon\Big\}.$$
When $N$ is large enough and $Q_N\cap Q_N'\cap  \{\max\sigma_i\in(1-\epsilon,1)\}$ holds, we have by \eqref{eq149}, \eqref{eqs139} and the fact $x_1\in(1,\frac{1}{\max\sigma_i})\subset(1,\frac{1}{1-\epsilon})$ that
\begin{multline}\label{eqs144}
	|\hat L_+-L_+|\le |1-\frac{1}{x_1}|+|\gamma_0\int\frac{t}{1-t}d\nu(t)-\frac{1}{N}\sum_{i=1}^M\frac{\sigma_i}{1-x_1\sigma_i}|
	\le\frac{\epsilon}{1-\epsilon}+|\gamma_0-\frac{M}{N}|\int\frac{t}{1-t}d\nu(t)\\
	+\frac{M}{N}\Big|\int\frac{t}{1-t}d\nu(t)-\frac{1}{  M}\sum\frac{\sigma_i}{1-\sigma_i}\Big|+\frac{M}{N}\Big|\frac{1}{  M}\sum\frac{\sigma_i}{1-\sigma_i}-\frac{1}{  M}\sum\frac{\sigma_i}{1-x_1\sigma_i}\Big|\\
	\le \frac{2\epsilon}{1-\epsilon}+N^{-\frac{1}{2}-c_0}\int\frac{t}{1-t}d\nu(t)+\frac{M}{N}\epsilon 
\end{multline}
where we used \eqref{eqn18} and \eqref{eqs143} in the last inequality.
Note that the RHS of \eqref{eqs144} goes to $\frac{2\epsilon}{1-\epsilon}+\gamma_0\epsilon$ and the probability of $Q_N\cap Q_N'\cap  \{\max\sigma_i\in(1-\epsilon,1)\}$ goes to 1. So
$$\P(|\hat L_+-L_+|\le\frac{2\epsilon}{1-\epsilon}+\gamma_0\epsilon+\epsilon)\to1. $$
Since $\epsilon$ can be arbitrarily small, we have for any (small) $\epsilon'>0$,
\begin{align}\label{eqs145}
	\P(|\hat L_+-L_+|<\epsilon')\to1
\end{align}
and \eqref{eq155} is proved under the assumption that $\gamma_0\ne \big(\int\big(\frac{t}{1-t}\big)^2d\nu(t)\big)^{-1}$.

So the first conclusion of Proposition \ref{prop_criteria} is proved.  The second conclusion can be proved similarly.
\end{proof}

\appendix

\section{Proof of some auxiliary lemmas}\label{proofs}

\begin{proof}[Proof of Lemma \ref{auxiliary1}]
\begin{enumerate}
	\item 
$\hat L_+<C_0$ is from Lemma 2.5 of \cite{Knowles+Yin}.  
By Lemma 2.4 and Lemma 2.5 of \cite{Knowles+Yin},
$$\hat L_-=f(x)$$
where 
$$f(t)=-\frac{1}{t}+\frac{1}{N}\sum_{i=1}^M\frac{\sigma_i}{1+t\sigma_i}$$
and $x$ is the unique critical point of $f$ on $(-\infty,-(\min\sigma_i)^{-1})\cup(0,\infty)$. By definition,
\begin{align}\label{eq6}
	\frac{N}{M}=\frac{1}{M}\sum_{i=1}^M\Big(\frac{x\sigma_i}{1+x\sigma_i}\Big)^2
\end{align}
so,
\begin{align*}
\hat L_-=f(x)=-\frac{1}{x}+\frac{\frac{1}{M}\sum_{i=1}^M\frac{\sigma_i}{1+x\sigma_i}}{\frac{1}{M}\sum_{i=1}^M\big(\frac{x\sigma_i}{1+x\sigma_i}\big)^2}=\frac{ \sum_{i=1}^M\frac{\sigma_i}{(1+x\sigma_i)^2}}{ \sum_{i=1}^M\frac{x^2\sigma_i^2}{(1+x\sigma_i)^2}}
\end{align*}
So it is enough to show $|x|$ is bounded above by a constant when $N>N_0$. But this is directly from \eqref{eq6} because $\frac{N}{M}\to\gamma_0^{-1}\ne1$.

\item Suppose $\lambda_1'\ge\cdots\ge\lambda_N'\ge0$ are eigenvalues of $XX^T$. By Theorem 2.10 of \cite{Bloemendal+Erdos+Knowles+Yau+Yin}, there exists 
a constant $C_1>1$ such that
\begin{align}\label{eq113}
	\P(\lambda_1'<C_1\quad\text{and}\quad\lambda_{M\wedge N}'>C_1^{-1})>1-N^{-D}
\end{align}
when $N$ is large enough. This together with the facts that $\lambda_1=\|\Sigma^{1/2}XX^T\Sigma^{1/2}\|\le \|XX^T\|=\lambda_1'$ yield:
\begin{align}\label{eq116}
\P(\lambda_1\le C_0)>1-N^{-D}
\end{align}
for $N>N_0(D)$.

To estimate $\lambda_{M\wedge N}$, we do singular value decomposition:
$$(O_1XO_2)_{ij}=\delta_{ij}\sqrt{\lambda_{ii}'}$$
where $O_1$ and $O_2$ are orthogonal matrices. 
Recall that if $W$ is a real symmetric matrix with nonnegative eigenvalues, then its smallest eigenvalues equals
$$\min_{\|v\|=1}\langle v,Wv^T\rangle.$$

If $M\le N$, then $\lambda_{M\wedge N}=\lambda_M$ is the smallest eigenvalue of $\Sigma^{1/2}XX^T\Sigma^{1/2}$ and 
\begin{multline}\label{eq114}
\lambda_{M\wedge N}=\min_{\|v\|=1}\langle v,\Sigma^{1/2}XX^T\Sigma^{1/2} v^T\rangle=\min_{\|v\|=1}\langle v,\Sigma^{1/2}O_1^T\text{Diag}(\lambda_1',\ldots,\lambda_M')O_1\Sigma^{1/2} v^T\rangle\\
=\min_{\|v\|=1}\langle u, \text{Diag}(\lambda_1',\ldots,\lambda_M') u^T\rangle\quad(\text{where }u=v\Sigma^{1/2}O_1^T )\\
\ge \min_{\|v\|=1}\|u\|\cdot \lambda_M'\ge \sqrt l\cdot \lambda_M'=\sqrt l\cdot \lambda_{M\wedge N}'.
\end{multline}
If $M> N$, then $\lambda_{M\wedge N}=\lambda_N$ is the smallest eigenvalue of $X^T \Sigma X$ and 
\begin{align}\label{eq115} 
	\lambda_{M\wedge N}=\min_{\|v\|=1}\langle v,X^T \Sigma X v^T\rangle\ge l\cdot\min_{\|v\|=1}\langle v,X^T  X v^T\rangle=l\cdot\lambda_N'=l\cdot \lambda_{M\wedge N}'
\end{align}
By \eqref{eq113}, \eqref{eq114} and \eqref{eq115},  
$$\P(\lambda_{M\wedge N}\ge C_0^{-1})>1-N^{-D}$$
for $N>N_0(D)$. This together with \eqref{eq116} complete the proof of \eqref{eq42}.
 
\end{enumerate}	
\end{proof}

\begin{proof}[Proof of Lemma \ref{lemma:preliminary}] 
	\begin{enumerate}
		\item \eqref{eq21} is from the definitions of $G_M$ and $G_N$. \eqref{eq97} is induced from \eqref{eq26} and \eqref{eq21}. \eqref{eq22}	can be proved similarly.
		\item \eqref{eq24} is   the Ward's identity (see (3.6) of \cite{Benaych-Georges+Knowles}). Now we use the argument on Page 283 of \cite{Knowles+Yin} to   prove \eqref{eq25}. For simplicity we let $S=\emptyset$. When $S\ne\emptyset$ the proof is the same. Consider the singular value decomposition:
		\begin{align}\label{eq27}
			\Sigma X=UAV
		\end{align}
		where $U$ is an $M\times M$ orthogonal matrix, $V$ is an $N\times N$ orthogonal matrix and $A$ is an $M\times N$
		matrix with $A_{rs}=\delta_{rs}\sqrt{\eta_s}$
		where   $\eta_1\ge\cdots\ge\eta_N\ge0$ are  eigenvalues of $X^T\Sigma^2 X$. By \eqref{eq26},
		\begin{multline}\label{eq11www}
			\sum_{j=1}^M|G_{ij}|^2=\sum_{j=1}^M|	(\Sigma XG_NX^T\Sigma)_{ij}-\Sigma_{ij}|^2\le 2\sum_{j=1}^M|(\Sigma XG_NX^T\Sigma)_{ij}|^2+2\sum_{j=1}^M(\Sigma_{ij})^2\\
			\le2(\Sigma XG_NX^T\Sigma^2 X\bar G_NX^T\Sigma)_{ii}+2(\Sigma^2)_{ii}\quad\text{(note: $G_N$ is symmetric)}\\
			=2(\Sigma XG_NV^TDiag(\eta_1,\ldots,\eta_N)V\bar G_NX^T\Sigma)_{ii}+2(\Sigma^2)_{ii}\quad\text{(by \eqref{eq27})} \\
			=2\sum_{k=1}^N\eta_k|(\Sigma XG_NV^T)_{ik}|^2+2(\Sigma^2)_{ii}
			\le 2\eta_1\sum_{k=1}^N|(\Sigma XG_N)_{ik}|^2+2(\Sigma^2)_{ii} \\
			=2\|X^T\Sigma^2 X\|\cdot(\Sigma X G_N\bar G_NX^T\Sigma)_{ii}+2(\Sigma^2)_{ii}
		\end{multline}
		where $\|\cdot\|$ denotes the operator norm. Since there exists orthogonal matrix $O$ such that $G_N=ODiag(\frac{1}{\lambda_1-z},\ldots,\frac{1}{\lambda_N-z})O^T$, we have
		\begin{multline}\label{eq10www}
			(\Sigma X G_N\bar G_NX^T\Sigma)_{ii}=\sum_{k=1}^N\frac{1}{|\lambda_k-z|^2}\big((\Sigma XO)_{ik}\big)^2=\frac{1}{\Im z}\sum_{k=1}^N\Im\big(\frac{1}{\lambda_k-z}\big)\big((\Sigma XO)_{ik}\big)^2\\
			=\frac{\Im(\Sigma XG_NX^T\Sigma)_{ii}}{\Im z}=\frac{\Im(G_M)_{ii}}{\Im z}\quad\text{(by \eqref{eq26})}
		\end{multline}
		
		Plug \eqref{eq10www} and the fact $\|X^T\Sigma^2X\|=\|\Sigma XX^T\Sigma\|\le \|XX^T\|\|\Sigma\|^2\le\|XX^T\|$ into \eqref{eq11www}, then we have
		\begin{align}\label{eq12www}
			\sum_{j=1}^M|G_{ij}|^2\le  \frac{2\|XX^T\|}{\Im z}(\Im(G_M)_{ii})+2
		\end{align} 
		
		Use singular value decomposition for $X$: $$X=O_1RO_2$$
		where $O_1$ and $O_2$ are orthogonal matrices, $R_{ij}=\delta_{ij}\sqrt{\mu_i}$ and $\mu_1\ge\cdots\ge\mu_N\ge0$ are eigenvalues of $X^TX$.  So 
		$$z^{-1}XX^T-\Sigma^{-1}=O_1Diag(z^{-1}\mu_i-\sigma_i^{-1})O_1^T.$$
		Then, by the definition of $G_M$, we have 
		$$|(G_M)_{ii}|=\big|\sum_{k=1}^M\frac{z\sigma_k}{\mu_k\sigma_k-z}((O_1)_{ik})^2\big|\le \frac{|z|}{|\Im z|}\sum_{k=1}^M((O_1)_{ik})^2=\frac{|z|}{|\Im z|}.$$
		This together with \eqref{eq12www} complete the proof of \eqref{eq25}.
		
		\item \eqref{eq28} is a direct corollary of  Theorem 2.10 in \cite{Bloemendal+Erdos+Knowles+Yau+Yin}.

		\item 	The first identity in	\eqref{eq12} is from \eqref{eq13} and \eqref{eq9}. The second identity in	\eqref{eq12} is from \eqref{eq15} and \eqref{eq9}. \eqref{eq18} is from \eqref{eq15}, \eqref{eq13} and the definitions of $Z_i$ and $Z_p$.
		\item 	\eqref{eq10} is (A.13) of \cite{Knowles+Yin}. The last identity in \eqref{eq11} can be shown by taking derivative on each term of
		$$\frac{1}{\hat m_{fc}}=-z+\frac{1}{N}\sum_{i=1}^M\frac{\sigma_i}{1+\sigma_i\hat m_{fc}}$$
		which comes from \eqref{self_consistent_eq}.
		
		\item 	
		The first identity of \eqref{eq14} is from \eqref{self_consistent_eq}. Now we prove the second identity of \eqref{eq14}.

		For any $p\in\{M+1,\ldots,M+N\}$, 
		\begin{align*}
			\frac{1}{G_{pp}}=\frac{1}{m_N}\Big[1+(\frac{m_N}{G_{pp}}-1)\Big]=\frac{1}{m_N}+\frac{1}{m_N}\frac{m_N-G_{pp}}{G_{pp}}=\frac{1}{m_N}+\frac{m_N-G_{pp}}{m_N^2}+\frac{(m_N-G_{pp})^2}{m_N^2G_{pp}}
		\end{align*}
		which (together with $Nm_N=\sum_{p=M+1}^{M+N}G_{pp}$) implies:
		\begin{align} \label{eq16}
			\frac{1}{N}\sum_{p=M+1}^{M+N}\frac{1}{G_{pp}}=\frac{1}{m_N}+\frac{1}{N}\sum_{p=M+1}^{M+N}\frac{(m_N-G_{pp})^2}{m_N^2G_{pp}}
		\end{align}
		By \eqref{eq12}, 
		\begin{multline} \label{eq17}
			\frac{1}{N}\sum_{i=1}^M \Big(G_{ii}+\frac{\sigma_i}{1+m_N\sigma_i}\Big)=\frac{1}{N}\sum_{i=1}^M \Big(\frac{-\sigma_i}{1+\sigma_im_N+\sigma_iZ_i-\sigma_iA_i}+\frac{\sigma_i}{1+m_N\sigma_i}\Big)\\
			=\frac{1}{N}\sum_{i=1}^M \Big(\frac{\sigma_i^2(Z_i-A_i)}{(1+\sigma_im_N+\sigma_iZ_i-\sigma_iA_i)(1+m_N\sigma_i)}\Big)\\
			=\frac{1}{N}\sum_{i=1}^M\frac{\sigma_i^2Z_i}{(1+\sigma_i m_N)^2}-\frac{1}{N}\sum_{i=1}^M\frac{\sigma_i^3Z_i(Z_i-A_i)}{(1+\sigma_i m_N)^2(1+\sigma_im_N+\sigma_iZ_i-\sigma_iA_i)}\\
			-\frac{1}{N}\sum_{i=1}^M\frac{\sigma_i^2A_i}{(1+\sigma_i m_N)(1+\sigma_im_N+\sigma_iZ_i-\sigma_iA_i)}\\
			=\frac{1}{N}\sum_{i=1}^M\frac{\sigma_i^2Z_i}{(1+\sigma_i m_N)^2}+B_N(z).
		\end{multline}
		
		By \eqref{eq12}, \eqref{eq16}, \eqref{eq17} and the definition of $h$, 
		\begin{multline*} 
			h(m_N)-z=-\frac{1}{N}\sum_{p=M+1}^{M+N}\frac{1}{G_{pp}}+\frac{1}{N}\sum_{p=M+1}^{M+N}\frac{(m_N-G_{pp})^2}{m_N^2G_{pp}}+\frac{1}{N}\sum_{i=1}^M\frac{\sigma_i}{1+m_N\sigma_i}-z\\
			=\frac{1}{N}\sum_{i=1}^M G_{ii}-\frac{1}{N}\sum_{p=M+1}^{M+N}A_p+\frac{1}{N}\sum_{p=M+1}^{M+N}Z_p+\frac{1}{N}\sum_{p=M+1}^{M+N}\frac{(m_N-G_{pp})^2}{m_N^2G_{pp}}+\frac{1}{N}\sum_{i=1}^M\frac{\sigma_i}{1+m_N\sigma_i}\\
			=\frac{1}{N}\sum_{i=1}^M \Big(G_{ii}+\frac{\sigma_i}{1+m_N\sigma_i}\Big)-\frac{1}{N}\sum_{p=M+1}^{M+N}A_p+\frac{1}{N}\sum_{p=M+1}^{M+N}Z_p+\frac{1}{N}\sum_{p=M+1}^{M+N}\frac{(m_N-G_{p})^2}{m_N^2G_{pp}}\\
			=\frac{1}{N}\sum_{i=1}^M\frac{\sigma_i^2Z_i}{(1+\sigma_i m_N)^2}+B_N(z)-\frac{1}{N}\sum_{p=M+1}^{M+N}A_p+\frac{1}{N}\sum_{p=M+1}^{M+N}Z_p+\frac{1}{N}\sum_{p=M+1}^{M+N}\frac{(m_N-G_{pp})^2}{m_N^2G_{pp}}
		\end{multline*}
		So we proved the second identity of \eqref{eq14}. 
	\end{enumerate}	
	
\end{proof}

 \section{A binary tree argument: proof of Lemma \ref{lemma:binary_tree}}\label{sec:binary_tree}

\begin{proof}
This  proof  follows the steps in the proof of Proposition 6.1 in \cite{Benaych-Georges+Knowles} with some slight modification. 	Suppose $k\in\{1,2,\ldots\}$ and $z$ is in 
\begin{align}\label{eq121}
\Big\{x+\i y\Big|\tau\le x\le \tau^{-1}, N^{-1+\tau}\le y\le\tau^{-1}\Big\}.
\end{align}
	 According to \eqref{eq18},
	\begin{align}\label{eq103}
		\E\Bigg[\mathds1_{z\in D_\tau\cup D_\tau'}\cdot\Bigg|\frac{1}{N}\sum_{p=M+1}^{M+N}Z_p\Bigg|^{2k}\Bigg]=\sum_{B\in\mathcal B_k}\frac{1}{N^{2k}}\sum_{\p}\E[\mathds1_{z\in D_\tau\cup D_\tau'}\cdot V(\p)] \cdot \mathds1(  g_\p=B)
	\end{align}
	where 
	\begin{itemize}
		\item $\mathcal B_k$ is the set of partitions of $\{1,\ldots,2k\}$;
		\item $\p=(p_1,\ldots,p_{2k})$ runs over $\{M+1,\ldots,M+N\}^{2k}$;
		\item $g_\p$ is the partition of $\{1,\ldots,2k\}$ induced by the coincidence of the coordinates of $\p$ (i.e., $a$ and $b$ are in the same block if and only if $p_a=p_b$);
		\item 
		\begin{multline*}
			V(\p)= Z_{p_1}\cdots Z_{p_k}\bar Z_{p_{k+1}}\cdots \bar Z_{p_{2k}} \\
			= (1-\E_{p_1})\Big[\frac{1}{G_{p_1p_1}}\Big]\cdots (1-\E_{p_k})\Big[\frac{1}{G_{p_kp_k}}\Big](1-\E_{p_{k+1}})\Big[\frac{1}{\bar G_{p_{k+1}p_{k+1}}}\Big]\cdots(1-\E_{p_{2k}})\Big[\frac{1}{\bar G_{p_{2k}p_{2k}}}\Big] 
		\end{multline*}
	\end{itemize}
	Given $B\in\mathcal B_k$, and $\p\in\{M+1,\ldots,M+N\}^{2k}$ with $g_\p=B$,
	\begin{itemize}
		\item we let $L(B)=\{r\in\{1,\ldots,2k\}|\text{the block in $B$ containing $r$ is \{r\}}\}$;
		\item we set $\p_L=\{p_r|r\in L(B)\}$;
		\item if  $\p_L\subset T\cup\{x,y\}$, then we say $G_{xy}^{(T)}$ is maximally expanded;
		\item set $\mathcal A$ to be the set of monomials in elements of
		\begin{align}
			\Big\{G_{xy}^{(T)}\Big|T\subset \p_L, x\ne y, \{x,y\}\subset\p\backslash T\Big\}\cup\Big\{\frac{1}{G_{xx}^{(T)}}\Big| T\subset \p_L, x\subset\p\backslash T\Big\}	
		\end{align}
		\item for each $A\in\mathcal A$, let $d(A)$ denote the number of off-diagonal entries (i.e., entries of form $G_{xy}^{(T)}$) in $A$.
	\end{itemize}
	
	Given $B\in\mathcal B_k$, and $\p\in\{M+1,\ldots,M+N\}^{2k}$ with $g_\p=B$, suppose $A\in\mathcal A$. We can do the following algorithm to expand $A$.
	\begin{enumerate}
		\item If every factor of $A$ is maximally expanded or $d(A) \ge 2k + 1$, then stop the algorithm.
		\item Otherwise choose a factor of $A$ that is not maximally expanded. If this entry is off-diagonal, $G_{xy}^{(T)}$, write
		\begin{align}\label{eq92}
			G_{xy}^{(T)}=G_{xy}^{(Tu)}+\frac{G_{xu}^{(T)}G_{uy}^{(T)}}{G_{uu}^{(T)}}.
		\end{align}	
		where $u$ is the smallest element in $\p_L\backslash(T\cup\{x,y\})$. If the chosen factor is diagonal, $\frac{1}{G_{xx}^{(T)}}$, then write
		\begin{align}\label{eq93}
			\frac{1}{G_{xx}^{(T)}}=\frac{1}{G_{xx}^{(Tu)}}-\frac{G_{xu}^{(T)}G_{ux}^{(T)}}{G_{xx}^{(T)}G_{uu}^{(T)}G_{xx}^{(Tu)}}
		\end{align}
		for the smallest $u\in\p_L\backslash(T\cup\{x,y\})$. Here \eqref{eq92} and \eqref{eq93} are induced from \eqref{eq9}. Now write $A=w_0(A)+w_1(A)$ by replacing the chosen factor by the RHS of \eqref{eq92} or \eqref{eq93}. Here $w_0(A)$ (resp. $w_1(A)$) denotes the terms containing the first therm on RHS of \eqref{eq92} (resp. \eqref{eq93}). Obviously
		\begin{align}
			d(w_0(A))=d(A),\quad d(w_1(A))\ge\max(2,d(A)+1).
		\end{align}
		
	\end{enumerate}
	
	Given $B\in\mathcal B_k$, and $\p\in\{M+1,\ldots,M+N\}^{2k}$ with $g_\p=B$, we apply the above algorithm on each 
	$$A^r:=\frac{1}{G_{p_rp_r}}\quad r\in \{1,\ldots,2k\}.$$
	We set $A_0^r:=w_0(A^r)$ and $A_1^r:=w_1(A^r)$. Then we apply the algorithm again and set
	$$A_{00}^r:=w_0(A_0^r),\quad A_{01}^r:=w_0(A_1^r),\quad A_{10}^r:=w_1(A_0^r),\quad A_{11}^r:=w_1(A_1^r)$$
	and so on. Notice that the lower indices are sequences of 0 and 1. So this algorithm generates a binary tree. After finite many steps, the algorithm stops and we have:
	$$A^r=\sum_{\sigma\in\mathcal L_r}A_\sigma^r\quad\text{and}\quad V(\p)=\sum_{\sigma_1\in\mathcal L_1}\cdots\sum_{\sigma_{2k}\in\mathcal L_{2k}} (1-\E_{p_1})[A_{\sigma_1}^1]\cdots (1-\E_{p_k})[A_{\sigma_k}^k](1-\E_{p_{k+1}})[\bar A_{\sigma_{k+1}}^{k+1}]\cdots (1-\E_{p_{2k}})[\bar A_{\sigma_{2k}}^{2k}] $$
	where $\mathcal L_r$ is the set of leaves (i.e., vertices with no child) in the binary tree generated by the algorithm. It is easy to see:
	\begin{itemize}
		\item $A_\sigma^r\in \mathcal A$ such that $d(A_\sigma^r)\ge 2k+1$ or every factor in $A_\sigma^r$ is maximally expanded;
		\item by the stopping rule of the algorithm, we have $d(A_\sigma^r)\le 2k+2$ for each $\sigma\in\mathcal L_r$;
		\item  each application of $w_1$ increases $d(\cdot)$ by at least one, and in the first step by two, so each $\sigma\in\mathcal L_r$ has no more than $2k+1$ ones;
		\item the number of Green function entries increases by at most four through each application of $w_1$ and by zero through each application of $w_0$, so $A_\sigma^r$ has no more than $8k+5$ Green function entries;
		\item each Green function entry $G_{xy}^{(T)}$ or $\frac{1}{G_{xx}^{(T)}}$ in $A_\sigma^r$ satisfies $T\subset\p_L$, so by $|\p_L|\le2k$, the total number of upper indices in $A_\sigma^r$ is no more than $2k(8k+5)$;
		\item each application of $w_0$ increases the total number of upper indices by one, so $\sigma$ has no more then $2k(8k+5)$ zeros;
		\item since each $\sigma\in\mathcal L_r$ has no more than $2k+1$ ones and no more than $2k(8k+5)$ zeros, we have $|\mathcal L_r|<C_k$ where $C_k>0$ depends only on $k$.
	\end{itemize}
	
	\begin{lemma}\label{lemma:two_cases} 
		Suppose $k\in\{1,2,\ldots\}$, 	 $B\in\mathcal B_k$ and  $z$ is in \eqref{eq121}. For any (small) $\epsilon'>0$, if $N>N_0=N_0(\epsilon',\tau,k)$, then:
		$$|\E[\mathds1_{z\in D_\tau\cup D_\tau'}\cdot V(\p)]|\le   N^{\epsilon'}|\Psi(z)|^{2k+|L(B)|}\quad\quad\text{provided }\p\in\{M+1,\ldots,M+N\}^{2k} \text{ and } g_\p=B.$$
	\end{lemma}
	\begin{proof}
		According to Lemma \ref{lemma:for_binary_tree}, for any $D'>0$, if $N>N_0(\epsilon',D',\tau,k)$ then
		\begin{align}\label{eq102}
			P(A_N)>1-N^{-D'}
		\end{align}
		where
		\begin{multline*}
			A_N=\Big\{|G_{ij}^{(T)}|\le N^{\epsilon'}\Psi(z)\text{ and }\frac{1}{|G_{ii}^{(T)}|}\le N^{\epsilon'} \quad\forall   i\ne j,  |T|\le 2k, z\in D_\tau\cup D_\tau'\Big\}\\
			\cap\Big\{ \Big|(1-\E_i)\frac{1}{G_{ii}^{(T)}}\Big|\le N^{\epsilon'}\Psi(z) \quad\forall  |T|\le 2k, z\in D_\tau\cup D_\tau'\Big\} 
		\end{multline*}

		Suppose $\p\in\{M+1,\ldots,M+N\}^{2k}$ and $g_\p=B$. Recall that 
		\begin{align}\label{eq94}
			V(\p)=\sum_{\sigma_1\in\mathcal L_1}\cdots\sum_{\sigma_{2k}\in\mathcal L_{2k}} (1-\E_{p_1})[A_{\sigma_1}^1]\cdots (1-\E_{p_k})[A_{\sigma_k}^k](1-\E_{p_{k+1}})[\bar A_{\sigma_{k+1}}^{k+1}]\cdots (1-\E_{p_{2k}})[\bar A_{\sigma_{2k}}^{2k}].
		\end{align}
		Recall that each $A_{\sigma_i}^i$ on RHS of \eqref{eq94} has no more than $8k+5$ Green function entries. Also note that if $\sigma_i$ has no one, then then $A_{\sigma_i}^i=\frac{1}{G_{ii}^{(T)}}$ for some  $T$, thus $d(A_{\sigma_i}^i)=0$. On the other hand, if  $\sigma_i$ has ones, then the number of ones is at most $d(A_{\sigma_i}^i)-1$ (since the first $w_1$ adds two off-diagonal entries). So  we have on $A_N$:
		\begin{multline}\label{eq96}
		\mathds1_{z\in D_\tau\cup D_\tau'}\cdot 	|(1-\E_i)[A_{\sigma_i}^i]|\le 
			\begin{cases}
				N^{\epsilon'}\Psi(z)\quad&\text{if }d(A_{\sigma_i}^i)=0\\
				2N^{\epsilon'(8k+5)}\cdot|N^{\epsilon'}\Psi(z)|^{d(A_{\sigma_i}^i)}\quad&\text{if }d(A_{\sigma_i}^i)>0
			\end{cases}\\
			\le 	2N^{\epsilon'(8k+5)}\cdot|N^{\epsilon'}\Psi(z)|^{1\vee d(A_{\sigma_i}^i)}\le 2N^{\epsilon'(8k+5)}\cdot|N^{\epsilon'}\Psi(z)|^{1+(\text{number of ones in }\sigma_i)}
		\end{multline}
		Now we consider two cases. First, if $A_{\sigma_r}^r$ $(1\le r\le 2k)$ contains a factor which is not maximally expanded, then by the algorithm, $d(A_{\sigma_r}^r)\ge 2k+1$.  So   if $N>N_0(\epsilon',D',\tau,k)$, then on $A_N$ we have
		\begin{multline}\label{eq100}
			\mathds1_{z\in D_\tau\cup D_\tau'}\cdot	\Big|(1-\E_{p_1})[A_{\sigma_1}^1]\cdots (1-\E_{p_k})[A_{\sigma_k}^k](1-\E_{p_{k+1}})[\bar A_{\sigma_{k+1}}^{k+1}]\cdots (1-\E_{p_{2k}})[\bar A_{\sigma_{2k}}^{2k}]\Big|\\
			\le N^{C_1\epsilon'}|\Psi(z)|^{4k}\le N^{C_1\epsilon'}|\Psi(z)|^{2k+|L(B)|}
		\end{multline}
		where $C_1>0$ is a constant depending only on $k$.
		
		Second, if every factor in each $A_{\sigma_i}^i$ ($1\le i\le 2k$) is maximally expanded and
		\begin{align}\label{eq95}
			\E\Big[	\mathds1_{z\in D_\tau\cup D_\tau'}\cdot\Big((1-\E_{p_1})[A_{\sigma_1}^1]\cdots (1-\E_{p_k})[A_{\sigma_k}^k](1-\E_{p_{k+1}})[\bar A_{\sigma_{k+1}}^{k+1}]\cdots (1-\E_{p_{2k}})[\bar A_{\sigma_{2k}}^{2k}]\Big)\Big] 
		\end{align}
		is nonzero, then we claim that for each $s\in L(B)$, there exists an element $a(s)\in\{1,\ldots,2k\}\backslash\{s\}$ such that $A_{\sigma_{a(s)}}^{a(s)}$ contains a Green function entry with lower index $p_s$.  This is because otherwise there exists $s\in L(B)$ such that for each $i\in\{1,\ldots,2k\}\backslash\{s\}$, $A_{\sigma_i}^i$ has no factor with lower index $p_s$, and therefore $A_{\sigma_i}^i$ must be $H^{(p_s)}$-measurable (since all factors of $A_{\sigma_i}^i$ are maximally expanded), so \eqref{eq95} equals:
	\begin{align}\label{eq159}
	\E\Big[(1-\E_{p_s})\Big[	\mathds1_{z\in D_\tau\cup D_\tau'}\cdot\prod_{i\in[1,k]\atop i\ne s}(1-\E_{p_i})[A_{\sigma_i}^i]\prod_{j\in[k+1,2k]\atop j\ne s}(1-\E_{p_j})[\bar A_{\sigma_j}^{j}]\Big] \Big]=0
	\end{align}	
		but this is contradictory to our assumption that \eqref{eq95} is nonzero. So the claim must be true.
		For each $i\in\{1,\ldots,2k\}$, if $l\in a^{-1}(i)$ (i.e., $l\in L(B)$ and $a(l)=i$), then by the definition of $L(B)$, $p_l$ is not the lower index of $A^i=\frac{1}{G_{p_ip_i}}$, so $p_l$ must be added to lower index of $A_{\sigma_i}^i$ through an application of $w_1$ (because $w_0$ does not change lower index while $w_1$ add one element in lower index). So $\sigma_i$ should have at least $|a^{-1}(i)|$ ones. So by \eqref{eq96},   if $N>N_0(\epsilon',D',\tau,k)$, then on $A_N$ we have
		\begin{multline}\label{eq101}
			\mathds1_{z\in D_\tau\cup D_\tau'}\cdot	\Big|(1-\E_{p_1})[A_{\sigma_1}^1]\cdots (1-\E_{p_k})[A_{\sigma_k}^k](1-\E_{p_{k+1}})[\bar A_{\sigma_{k+1}}^{k+1}]\cdots (1-\E_{p_{2k}})[\bar A_{\sigma_{2k}}^{2k}]\Big|\\
		\le N^{C_2\epsilon'}|\Psi(z)|^{2k+\sum_{i=1}^{2k}|a^{-1}(i)|}
			=N^{C_2\epsilon'}|\Psi(z)|^{2k+|L(B)|}
		\end{multline}
		where $C_2>0$ is a constant depending only on $k$.
		
		On $A_N^c$ we have the naive bound of the Green function entries:
		\begin{align}\label{eq98}
			|G_{xy}^{(T)}|\le N^2(1+\max_{a,b}|X_{ab}|)\quad\quad\quad\text{(by \eqref{eq22})}
		\end{align}
		Using \eqref{eq15} and \eqref{eq13} (with $G$ replaced by $G^{(T)}$) and \eqref{eq98} we have:
		\begin{align}\label{eq99} 
			\frac{1}{|G_{xx}^{(T)}|}\le l^{-1}+\tau^{-1}+(M\vee N)^4(1+\max_{a,b}|X_{ab}|)\max_{a,b}|X_{ab}|^2
		\end{align}
		Note that $\epsilon'$ can be arbitrarily small. Using \eqref{eq98} and \eqref{eq99}   on $A_N^c$, using \eqref{eq100} and \eqref{eq101} on $A_N$, using \eqref{eq102} to control $\P(A_N)$, using \eqref{eq94}, we complete the proof of Lemma \ref{lemma:two_cases}.
	\end{proof}
	Now we continue to prove Lemma \ref{lemma:binary_tree}. 
	Note that if $B\in\mathcal B_k$, then 
	\begin{align}\label{eq104} 
		\frac{1}{N^{2k}}\sum_{\p\in\{M+1,\ldots,M+N\}^{2k}}\mathds1(g_\p=B)\le N^{|B|-2k}\le N^{-k+\frac{1}{2}|L(B)|}=(N^{-1/2})^{2k-|L(B)|}
	\end{align}
	where we used $2k-|B|\ge\frac{2k-|L(B)|}{2}$ in the last inequality.
	
	By \eqref{eq103}, Lemma \ref{lemma:two_cases} and \eqref{eq104}, for any (small) $\epsilon'>0$, if $N>N_0=N_0(\epsilon',\tau,k)$, then for any $z$ in \eqref{eq121}  we have
	\begin{align*}  		\E\Big[\mathds1_{z\in D_\tau\cup D_\tau'}\cdot \Big|\frac{1}{N}\sum_{p=M+1}^{M+N}Z_p\Big|^{2k}\Big]\le |\mathcal B_k|N^{\epsilon'}|\Psi(z)|^{2k+|L(B)|}(N^{-1/2})^{2k-|L(B)|}\le|\mathcal B_k|N^{\epsilon'}|\Psi(z)|^{4k} 
	\end{align*}
	where we used $\Psi(z)\ge N^{-1/2}$ in the last inequality. For any $\epsilon>0$ and $D>0$, choose $k$ large enough such that $2k\epsilon>2D$, then for $N>N_0(\epsilon,D,\tau)$ we have:
	\begin{multline*}
		\P(\mathds1_{z\in D_\tau\cup D_\tau'}\cdot |\frac{1}{N}\sum_{p=M+1}^{M+N}Z_p|>N^\epsilon|\Psi(z)|^2)\le N^{-2k\epsilon}|\Psi(z)|^{-4k}\E[\mathds1_{z\in D_\tau\cup D_\tau'}\cdot |\frac{1}{N}\sum_{p=M+1}^{M+N}Z_p|^{2k}]\\
		\le|\mathcal B_k|N^\epsilon N^{-2k\epsilon}<N^{-D}
	\end{multline*}
	where we used the fact that $|\mathcal B_k|$ is a constant depending only on $k$ in the last inequality. 
	Finally, using the ``lattice" argument we complete the proof of the first conclusion.
	
	For the second conclusion, we notice: 
	\begin{align} 
		\E\Big[\mathds1_{z\in D_\tau\cup D_\tau'}\cdot\Big|\frac{1}{N}\sum_{i=1}^M\frac{\sigma_i^2}{(1+\sigma_i\hat m_{fc})^2}Z_i\Big|^{2k}\Big]=\sum_{B\in\mathcal B_k}\frac{1}{N^{2k}}\sum_{\p}\E[\mathds1_{z\in D_\tau\cup D_\tau'}\cdot\tilde V(\p) ]\cdot \mathds1(  g_\p=B)
	\end{align}
	where $\p$ runs over $\{1,\ldots,M\}^{2k}$ and 
	\begin{multline*}
		\tilde V(\p)= \Big(\frac{\sigma_{p_1}^2}{(1+\sigma_{p_1}\hat m_{fc})^2}Z_{p_1}\Big)\cdots \Big(\frac{\sigma_{p_k}^2}{(1+\sigma_{p_k}\hat m_{fc})^2}Z_{p_k}\Big)\Big(\frac{\sigma_{p_{k+1}}^2}{(1+\sigma_{p_{k+1}}\overline{\hat m_{fc}})^2}\bar Z_{p_{k+1}}\Big)\cdots \Big(\frac{\sigma_{p_{2k}}^2}{(1+\sigma_{p_{2k}}\overline{\hat m_{fc}})^2}\bar Z_{p_{2k}}\Big)
	\end{multline*}
	By $\eqref{eq70}$, $|\frac{\sigma_i^2}{(1+\sigma_i\hat m_{fc})^2}|$ is bound uniformly in $i\in\{1,\ldots,M\}$ and in $z\in D_\tau\cup D_\tau'$. We can use the same proof to show that the conclusion of Lemma \ref{lemma:two_cases} holds with $V(\p)$ replaced by $\tilde V(\p)$, provided $\p\in\{1,\ldots,M\}^{2k}$ and $g_\p=B$. The remaining part of the proof of the second conclusion is similar as that of the proof of the first conclusion.
\end{proof}


\begin{thebibliography}{99}  
	
	\bibitem{Bai+Silverstein} Zhidong Bai and Jack W. Silverstein, CLT for linear spectral statistics of large-dimensional sample covariance matrices, {\it Ann. Probab.} {\bf 32}, 2004.
	
	\bibitem{Bai+Silverstein_spectral_analysis} Zhidong Bai and Jack W. Silverstein, {\it Spectral Analysis of Large Dimensional Random Matrices}, Mathematics Monograph Series 2. Science Press,
	Beijing, 2006.
	
	\bibitem{Bai+Wang+Zhou} Z. D. Bai, X. Wang and W. Zhou, Functional CLT for sample covariance matrices, {\it Bernoulli} {\bf 16} (4), 1086–1113
 (2010).
 	
	\bibitem{Baik+Lee} 	J. Baik and J. O. Lee, Free energy of bipartite spherical Sherrington–Kirkpatrick model, {\it Ann. Inst. H. Poincar\'e Probab. Statist.} {\bf 56}, 2020.
	
\bibitem{Benaych-Georges+Knowles} Florent Benaych-Georges and Antti Knowles, Lectures on the local semicircle law for Wigner matrices, In {\it Advanced Topics in Random Matrices}, Panoramas et Synth\`eses {\bf 53}, Soci\'et\'e Math\'ematique de France, 2016.
	
	\bibitem{Bloemendal+Erdos+Knowles+Yau+Yin} 	A. Bloemendal, L. Erd\H{o}s, A. Knowles, H.-T. Yau and J. Yin, Isotropic local laws for sample covariance and generalized Wigner matrices, {\it Electron. J. Probab.} {\bf 19}, 2014.		
	
\bibitem{Durrett} R. Durrett, Probability: Theory and Examples, Ed 4.

\bibitem{Erdos+Yau} L. Erd{\H o}s and H.-T. Yau, {\it A Dynamical Approach	to Random Matrix Theory}, American Mathematical Society, Providence, 2017 
	
	\bibitem{Hachem+Hardy+Najim} W.	Hachem, A. Hardy, and J. Najim, Large complex correlated Wishart matrices: fluctuations and asymptotic
independence at the edges, \textit{The Annals of Probability} {\bf44} (2016).

\bibitem{Ji+Lee} Hong Chang Ji and Ji Oon Lee, Central limit theorem for linear spectral statistics of deformed Wigner matrices, {\it Random Matrices: Theory Appl.} {\bf 9}, 2020.

	\bibitem{Johansson} K. Johansson, Some limit theorems for the eigenvalues of a sample covariance matrix, {\it J. Multivariate Anal.} {bf 12}  (1), 1–38 (1982).

\bibitem{Knowles+Yin} 	Antti Knowles and Jun Yin, Anisotropic local laws for random matrices, {\it Probab. Theory Relat. Fields} {\bf 169}, 2017.

	
	\bibitem{Kwak+Lee+Park} 	J. Kwak, J. O. Lee and J. Park, Extremal eigenvalues of sample covariance
	matrices with general population, {\it Bernoulli} {\bf 27}, 2021.	

	\bibitem{Lee+Li} 	Ji Oon Lee and Yiting Li, Spherical Sherrington-Kirkpatrick model for deformed Wigner matrix
with fast decaying edges, {\it Journal of Statistical Physics} {\bf 190}, 2023.
	

	
	\bibitem{Lee+Schnelli} 	Ji Oon Lee and Kevin Schnelli, Tracy–Widom distribution for the largest eigenvalues of real sample covariance
	matrices with general population, {\it The Annals of Applied Probability} {\bf 26}, 2016.
	

	
	\bibitem{Li+Schnelli+Xu} 	Yiting Li, Kevin Schnelli and Yuanyuan Xu, Central limit theorem for mesoscopic eigenvalue statistics of	deformed Wigner matrices and sample covariance matrices, \textit{Ann. Inst. Henri Poincar\'e Probab. Stat.} {\bf57} (2021).
	
	\bibitem{Liu+Hu+Bai+Song} 	Zhijun Liu, Jiang Hu, Zhidong Bai, Haiyan Song, A CLT for the LSS of large dimensional sample covariance matrices with unbounded dispersions, arXiv:2205.07280, 2022.
	
	
	\bibitem{lytova+pastur} A. Lytova and L. Pastur, Central limit theorem for linear eigenvalue statistics of random
matrices with independent entries, {\it Ann. Prob.} {\bf 37}(5), 1778-1840 (2009).
	
	\bibitem{Marchenko+Pastur} V. A. Marchenko and L. Pastur, Distribution of eigenvalues for some sets of random matrices, {\it  Math. USSR, Sb.} {\bf 1} (4),  457–483 (1967).
	
	\bibitem{Najim+Yao} J. Najim and J. Yao, Gaussian fluctuations for linear spectral statistics of large random covariance matrices, {\it Ann. Appl. Probab.} {\bf 26} (3), 
	1837–1887 (2016).
	
	\bibitem{Pan+Zhou} G. Pan and W. Zhou,
Central limit theorem for signal-to-interference ratio of reduced rank linear receiver,	 {\it The Annals of Applied Probability} {\bf 18} (3),   1232–1270 (2008).
	
 
	 
 	\bibitem{M.Shcherbina} M. Shcherbina, Central limit theorem for linear eigenvalue statistics of the Wigner and sample
covariance random matrices, {\it Zh. Mat. Fiz. Anal. Geom.} {\bf 7}(2), 176-192 (2011).	 

	\bibitem{Silverstein+Choi} J. W. Silverstein and S. I. Choi, Analysis of the limiting spectral distribution of large dimensional random matrices, \textit{J. Multivariate Anal.} {\bf54} (1995).
	
	\bibitem{Voiculescu} D. Voiculescu, Multiplication of certain non-commuting random variables, {\it J. Operator Theory} {\bf 18} (2), 223–235 (1987).
	
	\bibitem{Voiculescu+Dykema+Nica} D. Voiculescu, K. J. Dykema and A. Nica, {\it Free Random Variables: A Noncommutative Probability Approach to Free Products with Applications to
		Random Matrices, Operator Algebras and Harmonic Analysis on Free Groups}, American Mathematical Society, Providence, 1992.


\end{thebibliography}
\end{document}